\documentclass{amsart}
\usepackage{amsmath}
\usepackage{amsfonts,amssymb,amsthm,mathrsfs}
\usepackage{mathscinet}
\usepackage{bm}
\usepackage{xcolor}
\usepackage{ulem}
\makeatletter
\newtheorem{theorem}{Theorem}[section]
\newtheorem{lemma}[theorem]{Lemma}
\newtheorem{proposition}[theorem]{Proposition}
\newtheorem{corollary}[theorem]{Corollary}

\theoremstyle{definition}

\newtheorem{definition}[theorem]{Definition}
\newtheorem{remark}[theorem]{Remark}
\newtheorem{example}[theorem]{Example}

\newcommand{\re}{\mathop{\mathrm{Re}}\nolimits}

\newcommand{\Sp}{\mathop{\mathrm{sp}}\nolimits}
\newcommand{\rank}{\mathop{\mathrm{rank}}\nolimits}
\newcommand{\card}{\mathop{\mathrm{card}}\nolimits}
\newcommand{\supp}{\mathop{\mathrm{supp}}\nolimits}
\newcommand{\1}{\mathbf{1}}

\newcommand{\CC}{\mathbb C}

\newcommand{\A}{\mathfrak{A}}
\newcommand{\B}{\mathfrak{B}}
\newcommand{\cB}{\mathcal{B}}
\newcommand{\cK}{\mathcal{K}}
\newcommand{\cH}{\mathcal{H}}
\newcommand{\cP}{\mathcal{P}}
\newcommand{\cN}{\mathcal{N}}

\def\PP{\mathbb P}
\newcommand{\perps}{\mathrel{\perp_{BJ}^s}}

\begin{document}
\title[Preservers of strong Birkhoff-James orthogonality]{On preservers of strong Birkhoff-James orthogonality between $C^*$-algebras}

\author{Bojan Kuzma}
\address[B. Kuzma]{${}^1$University of Primorska and ${}^{2}$Institute of Mathematics, Physics, and Mechanics}
\email{bojan.kuzma@upr.si}

\author{Srdjan Stefanovi\'{c}}
\address[S. Stefanovi\'{c}]{University of Belgrade, Faculty of Mathematics, Student\/ski trg 16-18, 11000 Belgrade, Serbia}
\email{srdjan.stefanovic@matf.bg.ac.rs}

\author{Ryotaro Tanaka}
\address[R. Tanaka]{Katsushika Division, Institute of Arts and Sciences, Tokyo University of Science, Tokyo 125-8585, Japan}
\email{r-tanaka@rs.tus.ac.jp}

\thanks{The work of the first author is supported in part by the Slovenian Research Agency (research program P1-0285 and research projects J1-50000 N1-0296 and N1-0428). The second author is supported by the Ministry of Science, Technological Development and Innovation of Republic of Serbia: grant number 451-03-47/2023-1/200104 with Faculty of Mathematics.
The third author was supported by JSPS KAKENHI Grant Number JP24K06788.}

\date{}
\begin{abstract}
It is shown that every linear strong Birkhoff-James isomorphism between unital $C^*$-algebras is a $*$-isomorphism followed by a unitary multiplication. Moreover, as a partial extension of this result to the non-unital case, the form of (possibly nonlinear) strong Birkhoff-James isomorphisms between compact $C^*$-algebras are determined. A nonlinear characterization of compact $C^*$-algebras in terms of strong Birkhoff-James orthogonality is also given.
\end{abstract}
\keywords{$C^*$-algebra, Compact $C^*$-algebra, Strong Birkhoff-James orthogonality,  Preservers.}
\subjclass{Primary 46L05; Secondary 46L35, 46B20, 46B80, 47B49.}
\maketitle
%%%%%%%%%%%%%%%%%%%%%%%%%%%%%%%%%%%%%%%%%%%%%%%%%%%%%%%%%%
%%%%%%%%%%%%%%%%%%%%%%%%%%%%%%%%%%%%%%%%%%%%%%%%%%%%%%%%%%
\section{Introduction}

An $\mathscr{S}$-preserver is a mapping that preserves a certain structure $\mathscr{S}$ of mathematical objects. As basic examples, a linear isomorphism is a linear-structure preserver and a homeomorphism is a topological-structure preserver. In most cases, preservers can be formulated as the class of (iso)morphisms in some concrete category, and the class of objects are chosen between those of Banach spaces and Hilbert $C^*$-modules depending on structures under consideration. It is important to obtain detailed descriptions of various preservers as barometers for amounts of information contained in the corresponding structures. In this direction, the finer the structure $\mathscr{S}$ is, the more restrictive $\mathscr{S}$-preservers are. A problem that asks for the form of $\mathscr{S}$-preservers is called a preserver problem of the structure $\mathscr{S}$.

The prototype of the study on preserver problems dates back to 1897, the Frobenius theorem: A linear determinant preserver on a matrix algebra $M_n (\mathbb{C})$ is either of the forms $A \mapsto QAP$ or $A \mapsto QA^TP$, where $P$ and $Q$ are regular matrices such that $\det (PQ)=1$. Roughly speaking, up to transpose, the matrix equivalence is controlled by the structure of determinant to some extent. If, further, a linear mapping preserves eigenvalues of matrices with multiplicity, then $Q$ can be chosen $P^{-1}$; see~\cite[Theorem 3]{MM59}. Here, it is worth noting that the structure of eigenvalues with multiplicity contains more information than that of determinant, since $\det (A)$ is the product of the eigenvalues of $A$ with multiplicity. Up to now, there are many interesting formulations and results on preserver problems concerning algebraic structures in matrix algebras, or more generally, in $C^*$-algebras. We refer to Hou and Petek-\v Semrl  \cite{Hou,Sem-Pet} for the structure of preservers of rank-one operators, which is one of the important reduction techniques in preserver problems.  Of similarly importance is reduction to singularity preservers~\cite{Dieudonne1949}  or   invertibility preservers; following the idea by Bre\v sar-\v Semrl~\cite{BresarSemrl1998} of using projections as an  attempt towards their classification, Aupetit~\cite{Aupetit2000} succeeded to show that unital linear bijections which preserve the invertibility  on  a (general)  von-Neumann  algebra, are Jordan maps.
Another powerful  technique in invertible linear preserver problems  on matrices is using overgroups. One can often easily find  the natural candidates which  preserve~$\mathscr{S}$ and it remains to check in the list which of their overgroups  still preserves~$\mathscr{S}$; for more on this   see {\DJ}okovi\' c, Li, and Platonov~\cite{DjokovicPlatonov1993,DjokovicLi1994,PlatonovDjokovic1995}. We refer also to Li-Pierce~\cite{LiPierce2001} for a brief overview of general techniques.   A different  approach is  using Hermitian operators (in a sense of Vidav-Palmer); this was used by Sourour~\cite{Sou} to determine that $X\mapsto UXV^\ast$ and $X\mapsto UX^tV^\ast$ ($U,V$ unitary) are all the isometries of  minimal normed ideals of $\cB(\ell_2)$, the algebra of bounded operators on a separable Hilbert space. One can also successfully use ring-theoretical techniques in solving (linear) preserver problems; we mention here the theory of functional identities. This was pioneered by Bre\v sar and its early success culminated in a solution~\cite{Bresar} of Herstein conjectures that  classify commutativity preserving maps on prime rings essentially as Jordan maps. Later it was fruitfully used by Alaminos, Bre\v sar, and  Villena in investigating Lie/Jordan derivations on von-Neumann algebras~\cite{Alaminos2004}.
General (unital) $C^*$-algebras lack projections or minimal elements, but one can exploit other properties in solving preserver problems. We mention here  Hatori and Moln\'ar's~\cite{HatoriMolnar2014} classification of isometries   of Thompson metric $\|\ln(a^{-1/2}ba^{-1/2})\|$ on a positive cone, from which one can further determine the structure of  positive homogeneous preservers isomorphisms of L\"owner order (we refer to \cite[Theorem~13]{Molnar2019} for a precise statement). By reducing to order-monotone maps,  preservers of various means on a positive cone of a general $C^*$-algebra were classified in~\cite{Dong2022}; see also~\cite{Molnar2022}.  The books by Bre\v{s}ar et al.~\cite{BCM07} and by Moln\'{a}r~\cite{Mol07} as well as the surveys by Bourhim and Mashreghi~\cite{BM15} and by Li and Tsing~\cite{LT92} exhibit plenty of interesting preserver problems on algebraic structures and well-developed techniques thereon.

Another important direction of preserver problems is formulated in the normed space settings by using norm-metrics. The simplest (but often highly nontrivial) one is a problem of determining the form of isometric isomorphisms between certain normed spaces. For example, the Banach-Stone theorem states that if $T\colon C_0(K) \to C_0(L)$ is an isometric isomorphism, then there exists a unimodular continuous function $u$ on $L$ and a homeomorphism $\varphi \colon K \to L$ such that $Tf=u(f \circ \varphi^{-1})$ for each $f \in C_0(K)$. In particular, the mapping $f \mapsto f \circ \varphi^{-1}$ preserves multiplication (and complex conjugation in the complex case). This means that the Banach space structure of $C_0(K)$ controls its whole structure as a Banach (or a $C^*$-) algebra. Moreover, the Banach-Stone theorem has the following generalization: Let $\A$ and $\B$ be $C^*$-algebras, and let $T\colon\A \to \B$ be an isometric isomorphism. Then there exists a unitary element $u \in \mathcal{M}(\B)$ and a Jordan $*$-isomorphism $J\colon\A \to \B$ such that $Tx=uJx$ for each $x \in \A$, where $\mathcal{M}(\B)$ is the multiplier algebra of $\B$; see Kadison~\cite{Kad51} for the unital case, and Paterson and Sinclair~\cite{PS72} for the non-unital case. Therefore, the Banach space structure of a $C^*$-algebra can still control its Jordan $*$-algebraic structure even if the algebra is non-commutative. Meanwhile, it is well-known as the Mazur-Ulam theorem that every (possibly nonlinear) surjective isometry between normed spaces is affine over the real numbers. Namely, as an affine mapping is a translation of a real-linear mapping, normed spaces are isometrically isomorphic as metric spaces if and only if they are isometrically isomorphic as real normed spaces. Hence, the (nonlinear) metric structure of a normed space solely determines its structure as a real normed space. The combination of the Kadison-Paterson-Sinclair theorem and the Mazur-Ulam theorem suggests a classification of $C^*$-algebras based on substructures of the metric structure.

Potential candidates are geometric structures determined by the shape of the unit ball of a normed space. The present paper focuses on the orthogonal structure with respect to (a strengthened version of) the Birkhoff-James (shortly, BJ) orthogonality which is a generalized orthogonality relation in normed spaces introduced by Birkhoff~\cite{Bir35}.

%%%%%%%%%%%%%%%%%%%%%%%%%%%%%%%%%%%%%%%%%%%%%%%%%%%%%%%%%%
\begin{definition}

Let $X$ be a normed space over $\mathbb{K}$, and let $x,y \in X$. Then, $x$ is said to be \emph{Birkhoff-James orthogonal} to $y$, denoted by $x \perp_{BJ} y$ if $\|x+\lambda y\| \geq \|x\|$ for each $\lambda \in \mathbb{K}$.
\end{definition}
%%%%%%%%%%%%%%%%%%%%%%%%%%%%%%%%%%%%%%%%%%%%%%%%%%%%%%%%%%

Basic properties of the BJ orthogonality were mainly discovered by James~\cite{Jam47a,James}. In common with the usual inner product orthogonality $\perp$, it is known that $\perp_{BJ}$ is non-degenerate and homogeneous, that is, $x \perp_{BJ}x$ implies $x=0$, and $x \perp_{BJ}y$ implies $\alpha x \perp_{BJ} \beta y$ for all scalars $\alpha$ and $\beta$. Meanwhile, one of the biggest difference between the usual orthogonality and BJ orthogonality is found in symmetry. Namely, if $\dim X \geq 3$, and if $x \perp_{BJ}y$ implies $y \perp_{BJ}x$, then $X$ is an inner product space. This result was shown independently by Day~\cite{Day47} and James~\cite{Jam47a}, originally for real normed spaces, based on a characterization of inner product spaces given by Kakutani~\cite{Kak39}. It can be also shown for complex normed spaces with the help of Bohnenblust~\cite{Boh42}; see also~\cite[Section 4]{Tan22a}. The readers interested in the BJ orthogonality are referred to comprehensive surveys~\cite{AMW12,AMW22}.

A linear preserver problem of the BJ orthogonality relation was considered by Koldobsky~\cite{Kol93} for the real case and Blanco and Turn\v{s}ek~\cite{BT06} for the general case, who came to the best conclusion. To be precise, if a linear mapping between normed spaces preserves the BJ orthogonality in one direction, then it is a scalar multiple of an isometry. This indicates that the BJ orthogonality structure contains much information. Indeed, it was shown in \cite{AGKRZ23} that every (possibly nonlinear) bijection between smooth normed spaces preserving the BJ orthogonality in both direction is essentially induced by an isometric isomorphism except for the two-dimensional case; see also \cite{IT22,Tan22c}. Moreover, it turned out that the commutativity of a $C^*$-algebra can be characterized in terms of the BJ orthogonality, and that a bijection between abelian $C^*$-algebras preserving the BJ orthogonality in both direction induces a $*$-isomorphism; see \cite{Tan22b,Tan23c}. Here, it should be emphasized that $C^*$-algebras are far from smooth except for $\mathbb{C}$. In this direction, recently, the study on nonlinear preserver problems on the BJ orthogonality for non-commutative $C^*$-algebras has also begun in \cite{KSi25a,KSi25b,KSt26}.

The Kadison-Paterson-Sinclair theorem reveals that the Jordan $*$-algebraic structure of a $C^*$-algebra obeys its Banach space structure, and the Koldobsky-Blanco-Turn\v{s}ek theorem ensures that the combination of linear and BJ orthogonal structures controls the whole Banach space structure. This simultaneously explains the robustness and limit of the BJ orthogonal structure of a $C^*$-algebra. In particular, to obtain the complete information of full $C^*$-algebraic structure, the BJ orthogonal structure is not enough. To overcome this difficulty, in fact, we already have an idea of fitting the BJ orthogonality to the algebraic structure of $C^*$-algebras. It is actually given for Hilbert $C^*$-modules in full generality.

For Hilbert $C^*$-modules (refer to Lance's standard textbook book~\cite{Lance1995}) one can accommodate the definition of BJ orthogonality to better fit the module structure. Instead of requiring that in a one-dimensional space $\CC y$   the origin comes closest to a given vector $x$, we require that this holds for a $C^*$-submodule generated by $y$. This was done in \cite{AR14} and is known henceforth as \emph{strong Birkhoff-James {\rm (shortly, strong BJ)}  orthogonality}. We will give its formal definition only for  $C^*$-algebras considered   naturally as  $C^*$-modules  over themselves, since this is the prime interest of our study.

%%%%%%%%%%%%%%%%%%%%%%%%%%%%%%%%%%%%%%%%%%%%%%%%%%%%%%%%%%
\begin{definition}\label{def:sBJ}
Let $\A$ be a $C^*$-algebra, and let $x,y \in \A$. Then, $x$ is said to be \emph{strongly  Birkhoff-James orthogonal} to $y$, denoted by $x \perps y$ if $\|x+yz\| \geq \|x\|$ for each $z \in \A$.
\end{definition}
%%%%%%%%%%%%%%%%%%%%%%%%%%%%%%%%%%%%%%%%%%%%%%%%%%%%%%%%%%

%%%%%%%%%%%%%%%%%%%%%%%%%%%%%%%%%%%%%%%%%%%%%%%%%%%%%%%%%%
\begin{remark}\label{rem:strong-implies-usual}
It is obvious that $x \perps y$ if and only if $x \perp_{BJ} yz$ for each $z \in \A$. Considering an approximate unit for $\A$, if necessary, one can show that $x \perps y$ implies $x \perp_{BJ}y$ (see~\cite[ implications in (2.1)]{AR14}).
\end{remark}
%%%%%%%%%%%%%%%%%%%%%%%%%%%%%%%%%%%%%%%%%%%%%%%%%%%%%%%%%%

%%%%%%%%%%%%%%%%%%%%%%%%%%%%%%%%%%%%%%%%%%%%%%%%%%%%%%%%%%
\begin{remark}\label{rem:subalgebra}
   Let $\A\subseteq\B$ be two $C^\ast$-algebras and $x,y\in\A$. Then, $x\perps y$ in $\A$ if and only if $x\perps y$ in  $\B$ (see \cite[Lemma~1.1]{KS23} or~\cite[Proposition 3.1]{AGKRTZ24} for an even stronger result).
\end{remark}
%%%%%%%%%%%%%%%%%%%%%%%%%%%%%%%%%%%%%%%%%%%%%%%%%%%%%%%%%%

As in the case of BJ orthogonality, the strong BJ orthogonality is non-degenerate and homogeneous, but not symmetric in general. Indeed, the strong BJ orthogonality is symmetric in a $C^*$-algebra $\A$ if and only if $\dim \A \leq 1$; see \cite[Corollary 2.7]{AR16}. Linear preserver problems on the strong BJ orthogonality were first studied by Arambasi\'{c} and Raji\'{c}~\cite{AR15}. It turned out that every linear bijection on $\cB(\cH)$ preserving the strong BJ orthogonality in both directions is a scalar multiple of a $*$-automorphism followed by a unitary element from the left. Later, in \cite[Theorem 5.7]{KST18}, this result was generalized to linear bijection between von Neumann algebras preserving the strong BJ orthogonality in both directions, which is currently the best in the existing results on the general linear case. In the context of nonlinear preservers, it is known that a bijection between $C_0(K)$ and $C_0(L)$ preserving the strong BJ orthogonality in both directions induces a homeomorphism between $K$ and $L$; see~\cite[Theorems 5.2 and 5.5]{TanakaContinuous}. Moreover, in \cite[Corollary 3.14]{AGKRTZ24}, it was shown that two compact $C^*$-algebras are $*$-isomorphic if (and only if) they admit a bijection preserving the strong BJ orthogonality in both directions. This is stated as a nonlinear classification result and does not contain any explicit description of preservers. In any (known) case, the strong BJ orthogonality structure distinguishes the full $C^*$-algebraic structure.

The present paper is devoted to preservers of strong BJ orthogonality between $C^*$-algebras which we assign the following term.
%%%%%%%%%%%%%%%%%%%%%%%%%%%%%%%%%%%%%%%%%%%%%%%%%%%%%%%%%%
\begin{definition}

Let $\A$ and $\B$ be $C^*$-algebras, and let $\Phi :\A \to \B$. Then $\Phi$ is called a \emph{strong BJ isomorphism} if it is bijective and $x \perp_{BJ}^s y$ is equivalent to $\Phi (x) \perp_{BJ}^s \Phi(y)$.
\end{definition}
%%%%%%%%%%%%%%%%%%%%%%%%%%%%%%%%%%%%%%%%%%%%%%%%%%%%%%%%%%
The main purpose of the present paper is twofold. One is to describe all linear strong BJ isomorphisms between unital $C^*$-algebras, which drastically improves the corresponding result for von Neumann algebras. As our method requires the existence of the unit for at least one of the $C^*$-algebras under consideration, the general (non-unital) case remains open. Another purpose is to partially complement  the general case by considering compact $C^*$-algebras.  We  begin by giving a  complete classification of their nonlinear strong BJ isomorphisms. We further   characterize  which $C^*$-algebras are compact $C^*$-algebras completely in terms of the strong BJ orthogonality structure.
This also substantially improves the existing results. Moreover, based on the nonlinear theory, all linear strong BJ isomorphisms between compact $C^*$-algebras are determined.   A fundamental observation, valid in both the unital and compact settings, is that every linear isometry which is also a strong BJ isomorphism arises from a \( * \)-isomorphism. This essentially follows from the decomposition theorem for Jordan \( * \)-isomorphisms between the (enveloping) von Neumann algebras into the direct sum of \( * \)-isomorphisms and \( * \)-anti-isomorphisms; see Kadison~\cite{Kad51} and St{\o}rmer~\cite{Stormer1965Jordan}. Consequently, the main problem in the classification of linear strong BJ isomorphisms reduces to determining whether they are isometric. In  the concluding remarks, a detailed description of linear strong BJ isomorphism on a compact $C^*$-algebra from the viewpoint of operator-theory is given.

%%%%%%%%%%%%%%%%%%%%%%%%%%%%%%%%%%%%%%%%%%%%%%%%%%%%%%%%%%
%%%%%%%%%%%%%%%%%%%%%%%%%%%%%%%%%%%%%%%%%%%%%%%%%%%%%%%%%%

%%%%%%%%%%%%%%%%%%%%%%%%%%%%%%%%%%%%%%%%%%%%%%%%%%%%%%%%%%
%%%%%%%%%%%%%%%%%%%%%%%%%%%%%%%%%%%%%%%%%%%%%%%%%%%%%%%%%%
\section{Preliminaries}

We follow the standard terminology in the theory of $C^*$-algebras; see, for example, \cite{Bla06,KR97a,KR97b,Murphy1990,Sak98}. Let $\A$ be a $C^*$-algebra. Then the symbols $\mathcal{S}(\A)$ and $\mathcal{P}(\A)$ stand for the set of all states and pure states of $\A$. The multiplicative unit for a unital $C^*$-algebra is denoted by $\1$. For notational convenience, the symbol~$\1$ will also be employed in non-unital algebras, in situations where its occurrence harmlessly  abbreviates expressions that do not presuppose the existence of a unit, such as such as $x(\1-y)^2x=x^2-2xyx+xy^2x$ or $e^x-\1=\sum_1^\infty x^n/n!$, etc. If $\A$ is non-unital, then the spectrum of an element of $\A$ is considered in the unitization $\widetilde{\A}$ of $\A$. It is known that each element of $\mathcal{S}(\A)$ (or $\mathcal{P}(\A)$) extends uniquely to an element of $\mathcal{S}(\widetilde{\A})$ (or $\mathcal{P}(\widetilde{\A})$). As usual, we always consider a non-unital $C^*$-algebra $\A$ as a closed two-sided ideal of $\widetilde{\A}$, and a (pure) state of $\A$ as a restriction of that of $\widetilde{\A}$. Regarding this convention, the only problem occurs in functional calculus of a self-adjoint element $x\in \A$, because the element $f(x)$ corresponding to $f \in C(\Sp (x))$ belongs to $\widetilde{\A}$ in general. However, we know that $f(x) \in \A$ whenever $f(0)=0$; for such an $f$ can be approximated by polynomials without constant terms. Let $|x|=(x^*x)^{1/2}$ for each $x \in \A$. Since $0^{1/2}=0$, we obtain $|x| \in \A$ regardless of whether $\A$ is unital or not.

For a (complex) Hilbert space $\cH$, let $\cB(\cH)$ and $\cK (\cH)$ denote the $C^*$-algebras of all bounded (linear) operators and compact operators on $\cH$. It is well-known that $\cK (\cH)$ is a closed two-sided ideal of $\cB (\cH)$. A $C^*$-subalgebra $\A$ of $\cB (\cH)$ is called a \emph{concrete} $C^*$-algebra; in which case, $\A$ is said to act on $\cH$. An element of a concrete $C^*$-algebra is often denoted by upper case letters as it is a bounded linear operator. By the Gelfand-Naimark theorem, every $C^*$-algebra is $*$-isomorphic to a concrete $C^*$-algebra.

A concrete $C^*$-algebra closed under the weak-operator topology is called a \emph{von Neumann algebra}. Recall that a net $(A_\lambda)_\lambda$ in $\cB (\cH)$ converges to $A \in \cB (\cH)$ with respect to the weak-operator topology if and only if $\lim_\lambda \langle A_\lambda \xi,\zeta \rangle = \langle A\xi,\zeta\rangle$ for each $\xi,\zeta \in \cH$; in which case, for any $B,C \in \cB (\cH)$, the net $(BA_\lambda C)_\lambda$ also converges to $BAC$ with respect to the weak-operator topology. Let $X$ be a subset of $\cB (\cH)$. Then the \emph{commutant} of $X$ is defined by $X' = \{ A \in \cB (\cH) : \text{$AB=BA$ for each $B \in X$}\}$. If $X$ is self-adjoint, that is, if $A \in X$ implies $A^* \in X$, then $X'$ is a von Neumann algebra acting on $\cH$.

A $C^*$-algebra $\A$ is said to be \emph{compact} if it is $*$-isomorphic to a $C^*$-subalgebra of $\cK (\cH)$. The definition adopted here is based on that of ``$C^*$-algebras of compact operators'' discussed in~\cite[Section 1.4]{Arveson}, although the term ``compact'' comes from \cite[Definition 3.1]{Ale68}. These two definitions are formally different from, but for $C^*$-algebras the same as, each other. This fact follows from the combination of \cite[Theorem 3.2, Corollary 3.3 and Lemma 3.4 (i)]{Ale68} and the structure theorem for compact $C^*$-algebras found in \cite[Theorem 8.2]{Ale68} and \cite[Theorem 1.4.5]{Arveson}.

\subsection{Approximate invertibility}
Approximately right invertible elements were  introduced in~\cite{Esmeral2023} and  play a major role in the context of nonunital $C^*$-algebra (in unital $C^*$-algebra, they are equivalent to right invertible elements; see  Proposition 2.7 and  the subsequent paragraph of~\cite{Esmeral2023}).

\begin{definition} Let $A$ be a $C^*$-algebra and let $a\in A$. We say that $a$ is approximately right invertible if there is a net $\{b_i\}_{i\in I}$ such that $ab_i$ is approximate unit for $A$.
\end{definition}

Characterization of approximately right invertible elements was given in \cite[Proposition 2.2]{KS23}.

\begin{proposition}[\cite{KS23}]\label{app.r.inv} Let $\A$ be   a (possibly nonunital) $C^*$-algebra. The following are equivalent for  $a\in \A$:
\begin{itemize}
    \item[(i)] $a$ is approximately right invertible.
    \item[(ii)] $a$ does not belong to any right modular ideal.
    \item[(iii)] For any pure state $\rho$ we have $\rho(aa^*)>0$.
\end{itemize}
\end{proposition}
We already mentioned that on unital $C^*$-algebras approximate and  usual right invertibility coincide. However, this does not translate to unitizations. Consider, e.g.,  a nonunital abelian $C^*$-algebra $\A=C_0(X)$ of all  complex-valued continuous functions  on $X=[0,1)$ that vanish at $1$. Here, pure states are point evaluations so, by the above classification, $f\in C_0(X)$ is approximately right invertible if and only if it has no zeros.  However, $f$, embedded in unitization~$\widetilde{\A}\simeq C([0,1])$, cannot be right invertible, because $f$ must vanish at $1$. Similarly, on $\A:=\cK(\cH)$, the $C^\ast$-algebra of all compact operators on some Hilbert space, pure states are vector states so $a\in\A$ is approximately right invertible if and only if its range is dense. In particular, the classical Volterra operator \( (Vf)(x)=\int_0^x f(t)\,dt \) is approximately right invertible, but clearly not right invertible in $\cB(\cH)$.
\subsection{Norm-classification of  strong BJ orthogonality}

\begin{lemma}\label{sBJ-char}

Let $\A$ be a $C^*$-algebra, and let $x,y \in \A$. Then, $x\perps  y$ if and only if $\|\|y\|^2x-yy^*x\|=\|x\|\|y\|^2$.
\end{lemma}
\begin{proof}
It suffices to consider the case $x,y \in \A \setminus \{0\}$. We note first that, in unitization~$\widetilde{\A}$ of~$\A$, $0 \leq \1-\|y\|^{-2}yy^* \leq \1$ which, in~$\A$, gives
 \begin{equation}\label{eq:two-ineqs}
x^* (\1 -\|y\|^{-2}yy^*)^2x \leq x^*(\1-\|y\|^{-2}yy^*)x \leq x^*x \leq \|x\|^2\1 ,
\end{equation}
and then implies that $\|(\1-\|y\|^{-2}yy^*)x\| \leq \|x\|$ or equivalently,
\begin{equation}\label{eq:inequality}
    \|\|y\|^2x-yy^*x\| \leq \|x\|\|y\|^2.
\end{equation}
If $x\perps  y$, then $x \perp_{BJ} y(y^*x)$ or
$\|x-\|y\|^{-2}yy^*x\| \geq \|x\|$. Hence, by~\eqref{eq:inequality}, $\|\|y\|^2x-yy^*x\| = \|x\|\|y\|^2$.

Conversely, suppose that $\|\|y\|^2x-yy^*x\| = \|x\|\|y\|^2$. Let $\rho \in \mathcal{S}(\A)$ be such that
\[
\rho (x^* (\1 -\|y\|^{-2}yy^*)^2x) = \|x^* (\1 -\|y\|^{-2}yy^*)^2x\| = \|x\|^2 .
\]
By~\eqref{eq:two-ineqs} it follows that $\rho (x^*(\1 -\|y\|^{-2}yy^*)x) = \rho (x^*x) =\|x\|^2$. In particular, we obtain
\[
\rho (x^*yy^*x ) = \|y\|\rho ( x^*x) - \|y\|\rho (x^*(\1 -\|y\|^{-1}yy^*)x) = 0 .
\]
Combining this with \cite[Theorem 2.5]{AR14}, we infer that $x\perps  y$.
\end{proof}

 \begin{lemma}\label{lem:reduction}
Let $\A$ be a  $C^*$-algebra,  $\rho \in \mathcal{S}(\A)$, and  $x \in \A$, $x\ge0$.  If $\rho (x)=\|x\|$, then $\rho (yx)=\rho(xy)=\rho(xyx)=\|x\|\rho (y) $ for each $y \in \A$. Meanwhile, if $\rho (x)=0$, then $\rho (y)=\rho ((\1-x)y(\1-x))$ for each $y \in \A$.
\end{lemma}
\begin{proof}
 We may assume $\|x\|=1$. Consider a unique extension of $\rho$ to the state of $\widetilde{\A}$. If $\rho (x)=1$, then since $0\le x^2\le x\le \1$, $1 =\rho (x) \leq \rho(x^2)^{1/2} \leq 1$, which implies that $\rho ((\1-x)^2)=0$. Hence,
\[
|\rho ((\1-x)y)| = |\rho (y^*(\1-x))| \leq \rho ((\1-x)^2)^{1/2}\rho (y^*y)^{1/2} =0
\]
for each $y \in \A$. This proves $\rho (y)=\rho (xy)=\rho (yx)$ and
 $\rho (y) = \rho (xyx)$ for each $y \in \A$.

Now, suppose that $\rho (x)=0$. In this case, we have $0 \leq \1-x$ and $\|\1-x\|=\rho (\1-x)=1$. Therefore, the preceding paragraph shows $\rho (y)=\rho ((\1-x)y(\1-x))$ for each $y \in \A$.
\end{proof}

\subsection{Reduction to positive cones}

%%%%%%%%%%%%%%%%%%%%%%%%%%%%%%%%%%%%%%%%%%%%%%%%%%%%%%%%%%
The  lemma below is known, see \cite[Lemma 1.2]{KS23} or, independently, \cite[Theorem 2.5]{AR14}. However, since it is crucial for us we present its short proof for the sake of convenience.
\begin{lemma}\label{abs-polar}

Let $\A$ be a $C^*$-algebra, and let $x,y \in \A$. Then, the following are equivalent:
 \begin{itemize}
\item[{\rm (i)}] $x\perps  y$.
\item[{\rm (ii)}] $|x^*|\perps  y$.
\item[{\rm (iii)}] $x\perps  |y^*|$.
\item[{\rm (iv)}] $|x^*|\perps  |y^*|$.
\end{itemize}
\end{lemma}
\begin{proof}
It is sufficient to prove (i) $\Leftrightarrow$ (iv). Considering the universal representation of $\A$, we may assume that $\A$ acts on a Hilbert space $\cH$. In this case, \cite[Proposition 3.1]{AGKRTZ24} allows us to use all operators on $\cH$ as right multiplier in $\perps$. Namely, $x\perps  y$ in $\A$ if and only if $x \perp_{BJ} yz$ for each $z \in \cB (\cH)$. Now, by polar decomposition, we recall that $x=|x^*|u$, $xu^*=|x^*|$, $y = |y^*|v$ and $yv^*=|y^*|$ for some partial isometries $u,v \in \cB (\cH)$.

Suppose that $x\perps  y$. Then, we have
\[
\||x^*|+|y^*|z\| \geq \|x+yv^*zu\| \geq \|x\|
\]
for each $z \in \cB (\cH)$, that is, $|x^*|\perps  |y^*|$. Conversely, if $|x^*|\perps  |y^*|$, then
\[
\|x+yz\| \geq \||x^*|+|y^*|vzu^* \| \geq \|x\|
\]
for each $z \in \cB (\cH)$. Therefore, $x\perps  y$.
\end{proof}
%%%%%%%%%%%%%%%%%%%%%%%%%%%%%%%%%%%%%%%%%%%%%%%%%%%%%%%%%

%%%%%%%%%%%%%%%%%%%%%%%%%%%%%%%%%%%%%%%%%%%%%%
%%%%%%%%%%%%%%%%%%%%%%%%%%%%%%%%%%%%%%%%%%%%%%%%%%%%%%%%%%

\subsection{Right symmetry}

Recall that $a\in\A$ is right-symmetric if $x\perps a$ implies $a\perps x$. The corresponding notion  of left-symmetricity also  exists, see \cite{GhoshSain2017} or  \cite{PaulMalWojcik2019}, but will not be needed here.  We characterize right symmetric elements via (pure) states, approximate invertibility, and  ortho-graphs. Recall~\cite[Definition 1.3]{KS23} that  ortho-graph (of a mutual strong BJ orthogonality relation), $\Gamma_{\mkern-7mu\perp\mkern-10mu\perp}(\A)$, of a $C^*$-algebra $\A$ is an undirected graph whose vertices are the all points of a projective space of $\mathbb{P}\A:=\{[x]=\CC x;\;\;x\in\A\setminus\{0\}\}$, with vertices $[x]$ and $[y]$ connected by an undirected edge if their representatives are mutually strong BJ orthogonal, i.e., $x\perps y$ and $y\perps x$. A vertex $[x]$ is isolated it no other vertex connects to it.
%%%%%%%%%%%%%%%%%%%%%%%%%%%%%%%%%%%%%%%%%%%%%%%%%%%%%%%%%%
%%%%%%%%%%%%%%%%%%%%%%%%%%%%%%%%%%%%%%%%%%%%%%%%%%%%%%%%%%
%%%%%%%%%%%%%%%%%%%%%%%%%%%%%%%%%%%%%%%%%%%%%%%%%%%%%%%%%%

%%%%%%%%%%%%%%%%%%%%%%%%%%%%%%%%%%%%%%%%%%%%%%%%%%%%%%%%%%

\begin{lemma}\label{right-symmetric}
Let $\A$ be a $C^*$-algebra. The following are equivalent for  $x \in \A \setminus \{0\}$.
\begin{itemize}
    \item[(i)]$x$ is right symmetric for $\perps $.
    \item[(ii)]  $y\perps  x$ implies $y=0$.
    \item[(iii)] $\rho (xx^*) >0$ for each $\rho \in \mathcal{P}(\A)$.
    \item[(iv)] $x$ is approximately right invertible.
    \item[(v)] $x$ is an isolated point in the ortho-graph $\Gamma_{\mkern-7mu\perp\mkern-10mu\perp}(\A)$.
     \end{itemize}
\end{lemma}
\begin{proof}
$\neg$(ii) $\Rightarrow$ $\neg$(i). Suppose that $y\perps  x$ for some $y \in \A \setminus \{0\}$. It may be assumed that $\|x\|=\|y\|=1$. By Lemma~\ref{sBJ-char}, we have $\|y-xx^*y\|=1$. Since $\|(\1-xx^*)yy^*(\1-xx^*)\|=\|(y-xx^*y)(y-xx^*y)^*\|=1$, there exists a state $\rho$ of $\A$ such that $\rho ( (\1-xx^*)yy^*(\1-xx^*)) = 1$. Extending $\rho$ to the unitization of $\A$ if necessary, we obtain
\[
1 = \rho ((\1-xx^*)yy^*(\1-xx^*)) \leq \rho ((\1-xx^*)^2) \leq \rho (\1-xx^*) = 1-\rho (xx^*).
\]
It follows that $\rho (xx^*) = 0$, which together with
\[
0 \leq \rho (|x^*|) \leq \rho (xx^*)^{1/2}\rho(\1)^{1/2}=0
\]
implies that $\rho (|x^*|)=0$. Therefore, the identity
\[
\rho ((\1-|x^*|)yy^*(\1 -|x^*|)) = \rho (yy^* ) = \rho ((\1-xx^*)yy^*(\1 -xx^*))=1
\]
holds. Now, let $z = |x^*| + (\1-|x^*|)yy^*(\1-|x^*|)$.  Since $\rho (z) =1$ and
\[
0 \leq z = |x^*| + (\1-|x^*|)yy^*(\1-|x^*|) \leq |x^*|+(\1-|x^*|)^2 \leq \1 ,
\]
it follows that $\|z\|=1$.
 Combining  with $|\rho (xx^*z)|\le\rho((xx^*)^*(xx^*))^{1/2}\rho(z^*z)^{1/2}\le \rho(xx^*)^{1/2}\rho(z^*z)^{1/2}=0$
we get $\|z+\lambda xx^* z\|\ge \rho(z+\lambda xx^* z)=1=\|z\|$  for every complex $\lambda$, so  $z \perp_{BJ} xx^*z$ (by definition of BJ orthogonality)
 which together with \cite[Theorem 2.5]{AR14} ensures $z\perps  x$.

We claim that $x \not\perps  z$. Indeed, if $x\perps  z$, then $|x^*|\perps |z^*|$ by Lemma~\ref{abs-polar}. We know that $|z^*|=z$,
\begin{align*}
|x^*|-2^{-1}z = 2^{-1}(|x^*|-(\1-|x^*|)yy^*(\1-|x^*|))
\end{align*}
and
\begin{align*}
-\1 &\leq -(\1-|x^*|)yy^*(\1-|x^*|)\\
&\leq |x^*|-(\1-|x^*|)yy^*(\1-|x^*|)\\
&\leq |x^*| \leq \1 .
\end{align*}
However, this means that $\||x^*|-2^{-1}z\| \leq 2^{-1}<1=\||x^*|\|$, that is, $|x^*| \not \perp_{BJ}z$. Thus, $x \not\perps z$, and $x$ is not right symmetric for $\perps $.

(ii) $\Rightarrow$ (i). Trivial observation.

(iii) $\Rightarrow$ (ii). Again it may be assumed that $\|x\|=1$. Suppose that $\rho (xx^*) >0$ for each pure state $\rho$ of $\A$. Let $y \in \A \setminus \{0\}$. Then, there exists a pure state $\rho$ of $\A$ such that $\rho ((\1 -xx^*)yy^*(\1-xx^*)) = \|(\1 -xx^*)y\|^2$. Recall   that $\rho$ extends uniquely to a pure state of the unitization of $\A$.
It follows from $\rho (xx^*) >0$ that  \[
\rho ((\1 -xx^*)yy^*(\1-xx^*)) \leq \|y\|^2\rho ((\1 -xx^*)^2) \leq \|y\|^2\rho (\1- xx^*) < \|y\|^2 ,
\]
which together with Lemma~\ref{sBJ-char} implies that $y \not\perps  x$.

$\neg$(iii) $\Rightarrow$ $\neg$(ii). We assume that $\rho (xx^*)=0$ for some pure state $\rho$ of $\A$. Then, the GNS representation $\pi_\rho \colon\A \to \cB(\cH_\rho)$ with respect to $\rho$ is irreducible. By the Kadison transitivity theorem,  there exists a self-adjoint element $h \in \A$ such that $\pi_\rho (h)\xi_\rho =\xi_\rho$, where $\xi_\rho$ is a unit  vector in $\cH_\rho$ such that $\rho (y) = \langle \pi_\rho (y)\xi_\rho ,\xi_\rho \rangle$ for each $y \in \A$.  Set $k = |\sin (2^{-1}h\pi)| \in \A$. It follows that $\pi_\rho (k)\xi_\rho = \xi_\rho$ and $\|k\|=\rho (k) = 1$.  Now, we note that
\[
0 \leq (\1-xx^*)k(\1-xx^*) \leq (\1-xx^*)^2 \leq \1-xx^* \leq \1
\]
and
\[
1 = \rho (k) =\rho ((\1-xx^*)k(\1-xx^*)) .
\]
Therefore, $\|(\1 -xx^*)k^{1/2}\| = \|(\1-xx^*)k(\1-xx^*)\|^{1/2}=1$. Combining this with Lemma~\ref{sBJ-char}, we infer that $k^{1/2}\perps  x$.

(iv) $\Leftrightarrow$ (iii). This is Proposition~\ref{app.r.inv}.

(iv) $\Leftrightarrow$ (v). See Propositions 2.5 and 2.6 in \cite{KS23}.
\end{proof}

Since in unital $C^*$-algebras, approximate right invertibility coincides  with  the usual right invertibility (see \cite[Proposition~2.7]{Esmeral2023}), we immediately get the next corollary.   However we decided to present a direct proof for convenience.
\begin{corollary}\label{cor:u-rinv}
Let $A$ be a unital $C^{*}$-algebra.  The following are equivalent for $x\in A\setminus\{0\}$.
\begin{itemize}
 \item[(i)] $x$ is right symmetric for the strong Birkhoff--James orthogonality $\perps$.
 \item[(ii)] $x$ is right invertible.
 \item[(iii)] $xx^\ast$ is invertible.
 \end{itemize}
  \end{corollary}
\begin{proof}
 Suppose that $x$ is not right invertible. In this case, $xx^*$ is not invertible; for otherwise, $x^*(xx^*)^{-1}$ is the right inverse for $x$. By  $0 \leq xx^* \leq \|xx^*\|\1$ and the spectral mapping theorem   we infer that $\| \1-\|xx^*\|^{-1}xx^*\| = 1$,  which together with Lemma~\ref{sBJ-char} implies that $\1\perps  x$. Meanwhile, we have $x \not\perps  \1$ by $x \neq 0$. Therefore, $x$ is not right symmetric for $\perps $.

If $x$ is right invertible with its right inverse $x'$, then $y\perps x$ implies that $y\perps  xx'y$. Hence, $y=0$ and $x\perps  y$. This shows  $x$ is right symmetric for~$\perps $.

Lastly, if $xx' = \1$ for some $x'$, then $\1 = (xx')(xx')^* = xx'(x')^*x^* \leq \|x'\|^2xx^*$, which ensures that $xx^*$ is invertible. The converse is also true as shown in the begining of the  proof. Hence, $x$ is right invertible if and only if $xx^*$ is invertible.
\end{proof}
%%%%%%%%%%%%%%%%%%%%%%%%%%%%%%%%%%%%%%%%%%%%%%%%%%

\section{Preservers}
\begin{definition}\label{Def:sBJ-iso}

Let $\A$ and $\B$ be $C^*$-algebras, and let $T\colon\A \to \B$. Then, $T$ is called a \emph{strong BJ isomorphism} if it is bijective and $x \perps y$ if and only if $Tx \perps Ty$.
\end{definition}

Clearly, every  linear  $\ast$-isomorphism  preserves the norm as well as the multiplication, so it is a  strong BJ isomorphism.  With  usual BJ orthogonality it suffices to assume less: namely,  linear isometries are automatically BJ isomorphisms; they, however,  might not preserve strong BJ orthogonality. Below we classify which of isometries are strong BJ isomorphism. We will rely on a general description of linear isometries on $C^*$-algebras: they are composed by a multiplication with a unitary element from a  multiplier algebra (which we briefly review next) and a Jordan $*$-isomorphism
(see~\cite{PS72} or \cite{Sherman}).

Given a $C^\ast$-algebra $\B$, represented nondegenerately on a Hilbert space~$\cH$ (i.e., $\B\xi=0$ implies $\xi=0$ for all $\xi\in\cH$), define
$$\mathcal{M}(\cB):=\{X\in\cB(\cH): X\B\subseteq \B\hbox{ and } \B X\subseteq\B\}.$$
This is a unital $C^*$-algebra, contained in the von Neumann envelope $\B''$ and is ($\ast$-isometrically isomorphic to) a \emph{multiplier algebra} of $\B$.  Moreover, $\mathcal{M}(\B) = \B$ if and only if $\B$ is unital. We refer to~\cite{Sherman} and references therein for more information.

%%%%%%%%%%%%%%%%%%%%%%%%%%%%%%%%%%%%%%%%%%%%%%%%%%%%%%%%%%
\begin{theorem}\label{linear-isometry}
Let $\A$ and $\B$ be $C^*$-algebras, and let $\Phi\colon  \A \to \B$ be  an isometry between them.   Then, $\Phi$ is a strong BJ isomorphism if and only if  there exist  unitary elements $u \in \mathcal{M}(\A)$ and $v \in \mathcal{M}(\B)$ and a $*$-isomorphism $\Psi\colon  \A \to \B$ such that $\Phi(x) = \Psi(u x)= v\Psi(x)$ for each $x \in \A$.
\begin{proof}
 We  only prove the nontrivial implication. Consider $\B$ as embedded into suitable $\cB(\cH)$. It is well-known (see~\cite{PS72} but see also \cite{Sherman}) that there exists a unitary $v\in\cB(\cH)$ (actually, $v\in {\mathcal M}(\B)$, the multiplier algebra of $\B$) such that $\Phi(x)=v J(x)$ for some Jordan $\ast$-isomorphism $J\colon \A\to\B$. Then also $v^\ast \in  {\mathcal M}(\B)$ and hence  $y\mapsto v^\ast y$ is a linear isometry on $\B$. It is easily seen, from the definition of a strong BJ orthogonality, that this map  is a strong BJ isomorphism. Thus, a Jordan $\ast$-isomorphism $J=v^\ast\Phi$  is a composite of two strong BJ isomorphisms. Now, it was proven by Kadison \cite[Theorem 2.6]{Kadison1965}
 and independently by St{\o}rmer in~\cite[Theorem~3.3]{Stormer1965Jordan}  that $\B\subseteq \B_1\oplus\B_2$ is (perhaps properly) contained in a direct sum of two $C^*$-algebras so that
$$J=J_1\oplus J_2,$$ with $J_1\colon\A\to \B_1$ being multiplicative and $J_2\colon\A\to\B_2$ anti-multiplicative (more precisely, there exists a central projection $P\in \B^{-}$, the weak-closure of $\B$ in $\cB(\cH)$, such that $J_1(x)=PJ(x)P$ is a compression of $J$, and $J_2(x)=QJ(x)Q$, $Q=(I-P)$).

Assume $\B_2$ is not abelian.  It is a classical result, due to Kaplansky, that then there exists a normalized $b=b_1\oplus b_2\in \B\subseteq \B_1\oplus \B_2$ such that $b_2^2=0$ (see, for example,~\cite[II.6.4.14]{Bla06}). Choose $a\in\A$ with $b=\Phi(a)=J_1(a)\oplus J_2(a)$.  Then,
\begin{align*}
0\oplus b_2b_2^*b_2&= 0\oplus J_2(aa^*a)
= 0\oplus (J_2(aa^*a)- J_2(a^*)b_2^2)\\
&= 0\oplus (J_2(aa^*a)- J_2(a^*)J_2(a)J_2(a))\\
&=(J_1(aa^*a)-J_1(aa^*)J_1(a))\oplus (J_2(aa^*a)- J_2(aa^*)J_2(a))\\
&=\Phi(aa^*a)-\Phi(aa^*)\Phi(a)\in\B.
\end{align*}
In particular,
$$y:=0\oplus b_2b_2^*b_2\in\B.$$
It is nonzero by $\|b^*y\|=\|0\oplus b_2^*b_2b_2^*b_2\|=\| (b_2^*b_2)^2\|=\|b_2\|^4\neq0$, however  $y^2=0$.

Now, $(y^*y)^*y=y^*y^2=0$, so $y\perp y^*y$. However, $y\in(0\oplus\B_2)\cap \B$, so if  $y=\Phi(x) =0\oplus J_2(x)$,  then  $$\Phi(xx^\ast)=J_1(x)J_1(x)^\ast\oplus J_2(x)^\ast J_2(x)=0\oplus J_2(x)^\ast J_2(x)=y^*y,$$
and $x\not\perps xx^\ast$, because  $|x^*| \not\perps |x^*|^2$ (see Lemma~~\ref{abs-polar} and~\ref{sBJ-char}).  This is a contradiction. Thus,  $\B_2$ is abelian, and  then $J_2$, hence also $J=J_1\oplus J_2$, is multiplicative, so a $\ast$-isomorphism.

This gives the first part since $\Phi(x)=v \Psi(x)$ for a $\ast$-isomorphism $\Psi=J$; the second part can be proven along the same lines starting with $\Phi^{-1}$.
\end{proof}
\end{theorem}

\subsection{The unital and linear case}\mbox{}

%%%%%%%%%%%%%%%%%%%%%%%%%%%%%%%%%%%%%%%%%%%%%%%%%%%%%%%%%%

The following is an extension of \cite[Theorem 5.4]{KST18}.
%%%%%%%%%%%%%%%%%%%%%%%%%%%%%%%%%%%%%%%%%%%%%%%%%%%%%%%%%%

\begin{lemma}\label{u-coisometry}
Let $\A$ be a $C^*$-algebra. The following are equivalent for  $x \in \A \setminus \{0\}$.
\begin{itemize}
    \item[(i)] $x\perps  y$ whenever $y$ is not right symmetric for $\perps $.
    \item[(ii)]  $\A$ is unital and $\|x\|^{-1}x$ is a coisometry.
\end{itemize}
\end{lemma}
\begin{proof}
It may be assumed that $\|x\|=1$. Suppose that $y \in \A$ is not right symmetric for $\perps $, in which case, $y$ is not right invertible by  Corollary~\ref{cor:u-rinv}.  As in the proof of Corollary~\ref{cor:u-rinv}, the spectral mapping theorem  gives $\| \1-\|yy^*\|^{-1}yy^*\| = 1$ which, by Lemma~\ref{sBJ-char}, implies that
$$\1\perps  y.$$  If $x$ is a coisometry, then $\1 = xx^* \in \A$, so also $|x^*|=\1\perps  y$ and hence $x\perps  y$, by Lemma~\ref{abs-polar}.

Conversely, suppose that $x$ is not a coisometry. Then, passing to unitization if needed, $|x^*| \neq \1$ and $\Sp (|x^*|) \subset [0,1]$ contains an element $t_0 \in [0,1)$. Let $f$ be a continuous function on $\Sp (|x^*|)$ such that $0 \leq f \leq \1$, $f(t_0)=0$ and $f(1)=1$. We note that
 the evaluation functional $\delta_{t_0}$ at $t_0$ corresponds to the restriction of a pure state $\rho \in \cP (\A)$ to the $C^*$-subalgebra generated by $|x^*|$. Hence, $y = |x^*|f(|x^*|)$ is not right symmetric for $\perps$ by  $\rho (y)=0$ and Lemma~\ref{right-symmetric}. However, we obtain $|x^*| \not \perps y$ by
\[
\||x^*|-y\| = \max_{t \in \Sp (|x^*|)}|t-tf(t)| < 1=\|\,|x^*|\,\| ,
\]
which implies that $x \not\perps  y$.
\end{proof}
%%%%%%%%%%%%%%%%%%%%%%%%%%%%%%%%%%%%%%%%%%%%%%%%%%%%%%%%%%
We can characterize unital $C^*$-algebras among all $C^*$-algebras by using strong Birkhoff-James orthogonality.
%%%%%%%%%%%%%%%%%%%%%%%%%%%%%%%%%%%%%%%%%%%%%%%%%%%%%%%%%%
\begin{proposition}\label{u-char}

Let $\A$ be a $C^*$-algebra. Then, $\A$ is unital if and only if there exists an element $x \in \A \setminus \{0\}$ such that $x\perps  y$ whenever $y$ is not right symmetric for $\perps $.
\begin{proof}
In the light of Lemma~\ref{u-coisometry}, the latter condition means the existence of a coisometry in $\A$. Combining this with the fact that $\A$ is unital if and only if it contains a coisometry, we obtain the proposition.
\end{proof}
\end{proposition}
%%%%%%%%%%%%%%%%%%%%%%%%%%%%%%%%%%%%%%%%%%%%%%%%%%%%%%%%%%
We need the following technical lemma.

%%%%%%%%%%%%%%%%%%%%%%%%%%%%%%%%%%%%%%%%%%%%%%%%%%%%%%%%%%
\begin{lemma}\label{u-rspec}

Let $\A$ be a unital $C^*$-algebra, and let $u \in \A$ be a coisometry. Then, there exists a complex number $\lambda$ such that $|\lambda |=1$ and $u-\lambda \1$ is not right invertible.
\begin{proof}
Since $u$ is a coisometry, $u^*$ is an isometry, which implies that $\|(u^*)^n\|=1$ for each $n$. It follows that the spectral radius of $u^*$ is $\lim_n \|(u^*)^n\|^{1/n}=1$. Let $\lambda \in \Sp (u^*)$ be such that $|\lambda|=1$. From now on, by GNS, we consider $\A$ as embedded into $\cB(\cH)$. We know that $\partial \Sp (u^*)$ is contained in the approximate point spectrum of $u^*$; see, for example,~\cite[Proposition VII.6.7]{Con90}. Hence, $\inf \{ \|(u^*-\lambda \1)(\xi) : \xi \in \cH,\|\xi\|=1\} =0$, which implies that $(u-\overline{\lambda}\1)^*$ is not left invertible, or equivalently, that $u-\overline{\lambda}\1$ is not right invertible.
\end{proof}
\end{lemma}
%%%%%%%%%%%%%%%%%%%%%%%%%%%%%%%%%%%%%%%%%%%%%%%%%%%%%%%%%%

%%%%%%%%%%%%%%%%%%%%%%%%%%%
The following is the main theorem in the unital and linear case.
%%%%%%%%%%%%%%%%%%%%%%%%%%%%%%%
\begin{lemma}\label{u-std}

Let $\A$ and $\B$ be $C^*$-algebras, and let $\Phi\colon \A \to \B$ be a linear strong BJ isomorphism.  Suppose that $\A$ is unital. Then, $\B$ is also unital and $\Phi$ is a scalar multiple of a $\ast$-isomorphism.
\begin{proof}
Since automorphism $\Phi$ preserves the right symmetry  it turns out from Proposition~\ref{u-char} that $\B$ is also unital. In this case, by    Corollary~\ref{cor:u-rinv} and Lemma~\ref{u-coisometry} $\Phi$ preserves right invertible elements and scalar multiples of coisometries. Set $\kappa = \|\Phi \1\|>0$.

Let $u$ be a coisometry in $\A$. Then, by Lemma~\ref{u-rspec}, $u-\lambda \1$ is not right invertible for some $\lambda \in \mathbb{C}$ with $|\lambda |=1$. By Lemma~\ref{u-coisometry}, it follows that $u\perps  u-\lambda \1$ and $\1\perps  u-\lambda \1$, which implies that $\Phi u\perps  \Phi u-\lambda \Phi \1$ and $\Phi \1\perps  \Phi u-\lambda \Phi \1$. In particular, we obtain
\[
\|\Phi u\| \leq \| \Phi u -(\Phi u -\lambda \Phi \1)\| = \|\Phi \1\|
\]
and
\[
\|\Phi \1 \| \leq \|\Phi \1 +\lambda^{-1}(\Phi u-\lambda \Phi \1)\| =\|\Phi u\| ,
\]
that is, $\|\Phi u\|=\|\Phi \1\|=\kappa$. If $v \in \B$ is a coisometry, then $\|\Phi^{-1}v\|^{-1}\Phi^{-1}v$ is a coisometry in $\A$ and $\kappa =\|\Phi (\|\Phi^{-1}v\|^{-1}\Phi^{-1}v)\| = \|\Phi^{-1}v\|^{-1}$, that is, $\kappa v = \Phi (\|\Phi^{-1}v\|^{-1}\Phi^{-1}v)$. From this, setting $\Psi = \kappa^{-1}\Phi$ yields a linear strong BJ isomorphism
 that preserves coisometries in both directions. In particular, $\|\Psi(u)\|=1$ for each unitary $u$.
Now, by the Russo-Dye theorem~\cite{Gardner1984}, each element  of the open unit ball $B_\A$ is a convex combination of unitary elements of $\A$. Combining this with $\|\Psi (u)\|=1$ for each unitary  $u \in \A$, we get $\|\Psi(r \tfrac{x}{\|x\|}) \| \leq 1$ for every nonzero $x\in\A$ and $r\in(0,1)$.  Equivalently, $\|\Psi(x)\|\le\|x\|$. Applying  the same arguments to linear strong BJ isomorphism  $\Psi^{-1}$ gives that $\Psi$ is an isometric isomorphism. The rest follows from Theorem~\ref{linear-isometry}.
 \end{proof}
\end{lemma}
%%%%%%%%%%%%%%%%%%%%%%%%%%%%%%%%%%%%%%%%%%%%%%%%%%%%%%%%%%

%%%%%%%%%%%%%%%%%%%%%%%%%%%%%%%%%%%%%%%%%%%%%%%%%%%%%%%%%%

\subsection{Compact operators}

For each unit vector $\xi \in \cH$, let $E_\xi$ be the rank-one projection onto $\mathbb{C}\xi$.

%%%%%%%%%%%%%%%%%%%%%%%%%%%%%%%%%%%%%%%%%%%%%%%%%%%%%%%%%
\begin{remark}

Let $\A$ be a compact $C^*$-algebra. Then, it is $*$-isomorphic to a $C^*$-algebra of the form $(\bigoplus_{\lambda \in \Lambda}  \cK (\cH_\lambda ))_{c_0(\Lambda )}$, which can be viewed as a $C^*$-subalgebra of $\cK (\cH)$, where $\cH = \bigoplus_{\lambda \in \Lambda} \cH_\lambda$ (see, e.g.,  \cite[Theorem 1.4.5]{Arveson}). Let $P_\lambda$ be the projection from $\cH$ onto $\cH_\lambda$ and let $\delta_{\lambda \mu} $ denote the Kronecker~$\delta$. Then, we have $P_\mu (A_\lambda)_{\lambda \in \Lambda} = (\delta_{\lambda \mu} A_\lambda )_{\lambda \in \Lambda} = (A_\lambda )_{\lambda \in \Lambda}P_\mu$ for each $(A_\lambda )_{\lambda \in \Lambda} \in \A$ and each $\mu \in \Lambda$. It follows that $\A \subset \{ P_\lambda :\lambda \in \Lambda \}' \cap \cK (\cH)$. Conversely, if $A \in \{ P_\lambda :\lambda \in \Lambda \}' \cap \cK (\cH)$, then $AP_\lambda = P_\lambda AP_\lambda$ which can be viewed as an element of $\cK (\cH_\lambda )$. Moreover,
  we have  $A = \sum_{\lambda \in \Lambda}AP_\lambda$ in the norm topology --- this easily follows for  rank-one $A$, then for finite-ranks and they   are norm-dense in $\A$.
 Therefore, $\A \supset \{ P_\lambda :\lambda \in \Lambda \}' \cap \cK (\cH)$, that is, $\A =\{ P_\lambda :\lambda \in \Lambda \}' \cap \cK (\cH)$. Consequently, if $\A$ is a compact $C^*$-algebra, then we may assume the following:
\begin{itemize}
\item[(i)] $\A$ acts on a Hilbert space $\cH$.
\item[(ii)] There exists an orthogonal family of projections $(P_\lambda )_{\lambda \in \Lambda} \subset \cB (\cH)$ such that $\sum_{\lambda \in \Lambda}P_\lambda =I$ in the strong-operator topology and
\[
\A = \{ P_\lambda : \lambda \in \Lambda \}' \cap \cK (\cH).
\]
\end{itemize}
In this case, we note that $\A P_\lambda = P_\lambda \A P_\lambda =P_\lambda \cK (\cH)P_\lambda \subset \A$ for each $\lambda \in \Lambda$.
\end{remark}
%%%%%%%%%%%%%%%%%%%%%%%%%%%%%%%%%%%%%%%%%%%%%%%%%%%%%%%%%
 Given $A\in\cB(\cH)$ we let
$$M_A:=\{\xi\in\cH: \|A\xi\|=\|A\|\cdot\|\xi\|\}.$$ Note that for a compact operator $A$ this set is always nonempty.
 Clearly, if $A=0$, then $M_A=\cH$. However, if $\|A\|=1$, then $\xi \in M_A$ means that
\[
\|\xi\|^2 = \|A\xi\|^2 =\langle |A|^2\xi,\xi \rangle \leq \langle |A|\xi,\xi \rangle \leq \|\xi\|^2 .
\]
Hence, $\langle |A|\xi,\xi \rangle =\|\xi\|^2 = \|A\|\|\xi\|^2$, which together with the equality condition for the Cauchy-Schwarz inequality implies that $|A|\xi = \xi$. Since $M_{\gamma A} = M_A$ whenever $\gamma \neq0$, it follows that
\[
M_A = \{ \xi \in \cH : |A|\xi = \|A\|\xi\} = \ker (\|A\|\1-|A|)=\ker (\|A\|^2\1-A^*A).
\]
In particular, $M_A$ is a (possibly trivial) closed subspace of $\cH$.

 Our next aim is to describe the pure states of a general compact $C^*$-algebra and  their connection to strong BJ orthogonality (see Lemma~\ref{sBJ-pure}).  Given a normalized vector $\xi\in\cH$, we let $\omega_\xi$ be the induced vector state on $\cB(\cH)$; it is given by $a\mapsto\langle a\xi,\xi\rangle\in\CC$. In $C^*$-algebra $\cK(\cH)$ these are the only pure states. In general $C^*$-algebra $\A$ of (not necessarily all) compact operators we show first  that $\xi$ must belong to a single block-component of $\cH=\bigoplus \cH_{\lambda}$.
\begin{lemma}\label{cpt-pure-state}

Let $\cH$ be a Hilbert space, let $(P_\lambda )_{\lambda \in \Lambda} \subset \cB (\cH)$ be an orthogonal family of projections with sum $I$, and let $\A = \{ P_\lambda : \lambda \in \Lambda \}' \cap \cK (\cH)$. Then,
\[
\cP (\A) = \{ \omega_\xi : \xi \in \cH,~\|\xi\|=1,~\text{$P_\lambda \xi = \xi$ for some $\lambda \in \Lambda$}\} .
\]
\begin{proof}
Let $\rho \in \cP (\A)$. Since $\A$ is a $C^*$-subalgebra of $\cK (\cH)$, each pure state of $\A$ extends to a pure state of $\cK (\cH)$. Let $\xi \in \cH$ be a unit vector such that $\rho = \omega_\xi$, and let $$\Lambda_\xi= \{ \lambda \in \Lambda : P_\lambda \xi \neq 0\}.$$ It follows from $\sum_{\lambda \in \Lambda} P_\lambda =I$ that $\Lambda_\xi \neq \emptyset$. To show that $\Lambda_\xi$ is a singleton, suppose to the contrary that $\Lambda_\xi$ contains two distinct elements $\lambda,\mu$. In this case, $0 \neq P_\lambda \xi \neq \xi$ and $0<\|P_\lambda \xi\| <1$ by $\|P_\lambda \xi\|^2+\|(I-P_\lambda )\xi\|^2=1$. Set $\zeta =\|P_\lambda \xi\|^{-1}P_\lambda\xi$ and $\eta = \|(I-P_\lambda )\xi\|^{-1}(I-P_\lambda )\xi$. Since $\A \subset  \{ P_\lambda :\lambda \in \Lambda \}'$, we obtain
\begin{align*}
(\|P_\lambda \xi\|^2\omega_\zeta + \|(I-P_\lambda )\xi\|^2\omega_\eta)(A)
&= \langle AP_\lambda \xi ,P_\lambda \xi \rangle + \langle A(I-P_\lambda )\xi,(I-P_\lambda )\xi\rangle \\
&= \langle P_\lambda A\xi,\xi \rangle +\langle (I-P_\lambda )A\xi,\xi \rangle \\
&= \omega_\xi (A)
\end{align*}
for each $A \in \A$, which implies that $\|P_\lambda \xi\|^2\omega_\zeta = \|(I-P_\lambda )\xi\|^2\omega_\eta$. However, this is impossible by $E_\zeta = P_\lambda E_\zeta P_\lambda \in \A$, where $E_\zeta$ is a rank-one projection onto $\mathbb{C}\zeta$.
 Hence, $\Lambda_\xi$ is a singleton $\{\lambda\}$, and then $P_\lambda \xi = \xi$.

Conversely, we assume that $\xi \in \cH$, $\|\xi\|=1$ and $P_\lambda \xi = \xi$ for some $\lambda \in \Lambda$. Then, $E_\xi = P_\lambda E_\xi P_\lambda \in \A$. Suppose that $\rho ,\tau \in \mathcal{S}(\A)$, $t \in (0,1)$ and $\omega_\xi = (1-t)\rho +t\tau$. It follows from $\omega_\xi (E_\xi)=1$ that $\rho (E_\xi)=\tau (E_\xi)=1$, which,  by Lemma~\ref{lem:reduction},
that $\rho (A) = \rho (E_\xi AE_\xi )$ and $\tau (A) = \tau (E_\xi AE_\xi )$ for each $A \in \A$. Since $E_\xi AE_\xi = \omega_\xi (A)E_\xi$, we have $\omega_\xi = \rho = \tau$. Therefore, $\omega_\xi \in \cP (\A)$.
\end{proof}
\end{lemma}
%%%%%%%%%%%%%%%%%%%%%%%%%%%%%%%%%%%%%%%%%%%%%%%%%%%%%%%%%
Given a nonzero $x$ in a $C^*$-algebra $\A$, let
$$    \cP_x = \{ \rho \in \cP (\A) : \rho (|x|)=\|x\|\}.$$
%%%%%%%%%%%%%%%%%%%%%%%%%%%%%%%%%%%%%%%%%%%%%%%%%%%%%%%%%
Notice that $ \cP_x= \cP_{|x|}$. It was proven in \cite[Theorem 2.5]{AR14} that $x\perps y$ if and only if there exists a state $\rho$ such that $\rho(x^\ast x)=\|x\|^2$ and $\rho(x^\ast yy^\ast x)=0$. We can now show that, to describe strong BJ orthogonality on  compact $C^*$-algebras,  pure states suffice. Remark that such a description is not possible with the usual BJ. For example, in $\A=\cK(\cH)\oplus\cK(\cH)$, where $\cH$ is two-dimensional, we have $(\1\oplus \1)\perp ((-\1)\oplus \1)$, because there is a state $\phi=1/2(\omega_{\xi\oplus 0}+\omega_{0\oplus \xi})$ with $\phi(\1\oplus \1)=1$ and $\phi((-\1)\oplus \1)=0$; however, there does not exist a pure state $\rho$ with $\rho((-\1)\oplus \1)=0$ (compare with \cite[Theorem 2.2]{AR14}).
\begin{lemma}\label{sBJ-pure}

Let $\A$ be a   $C^*$-algebra, and let $x,y \in \A$. Suppose that $x \neq 0$. Then, $x\perps  y$ if and only if there exists a $\rho \in \cP_{|x^*|}$ such that $\rho (|y^*|)=0$.
\end{lemma}
\begin{proof}
   Suppose that $x \perp_{BJ}^s y$. It may be assumed that $\A$ is unital (see~Remark~\ref{rem:subalgebra}) and $\|x\|=\|y\|=1$.
In view of Lemma~\ref{abs-polar} and Remark~\ref{rem:strong-implies-usual}, $|x^*| \perp_{BJ} |y^*|$, and then,  by James's result~\cite[Theorem~2.1]{James}, there exists a (supporting) functional $\rho \in \A^*$ such that $\|\rho\|=\rho (|x^*|)=1$ and $\rho (|y^*|)=0$. Define the functional $\rho^*$ on $\A$ by $\rho^*(z)=\overline{\rho (z^*)}$. It follows that $\rho^* \in \A^*$, $\|\rho^*\|=\rho^*(|x^*|)=1$ and $\rho^*(|y^*|)=0$. Replacing $\rho$ with $2^{-1}(\rho+\rho^*)$ if necessary, we may assume that $\rho^*=\rho$. In this case, there exists a unique pair of positive functionals $(\rho^+,\rho^-)$ on $\A$ such that $\rho = \rho^+-\rho-$ and $\|\rho\|=\|\rho^+\|+\|\rho^-\|$; see, for example,~\cite[Theorem 4.3.6]{KR97a}. Since
\[
1=\rho (|x^*|)=\rho^+(|x^*|)-\rho^-(|x^*|) \leq \rho^+(|x^*|) \leq \|\rho^+\| \leq \|\rho\|=1,
\]
we obtain $\|\rho^+\|=\|\rho\|=1$ and $\rho^-=0$. Hence, $\rho = \rho^+ \in \mathcal{S}(\A)$. Now, let $F=\{ \tau \in \mathcal{S}(\A) : \tau (|x^*|)=1,~\tau(|y^*|)=0\}$. Then, $F$ is a nonempty weakly$^*$ compact face of $\mathcal{S}(\A)$. By the Krein-Milman theorem, we have an extreme point $\tau$ of $F$,  which has the desired property.

Conversely, suppose that $\rho (|y^*|)=0$ for some $\rho \in \mathcal{P}_{|x^*|}$. (In fact, $\rho \in \mathcal{S}(\A)$ is sufficient.) It follows from $|y^*|^2 = |y^*|^{1/2}|y^*||y^*|^{1/2} \leq \|y\|\cdot|y^*|$ that $\rho (|y^*|^2)=0$. Then, by Lemma~\ref{lem:reduction}
\[|\rho (|x^*||y^*|\cdot z)| = \|x\||\rho (|y^*|z)| \leq \|x\|\rho (|z|^2)^{1/2}\rho(|y^*|^2)^{1/2}=0
\]
for each $z \in \A$. Thus, $|x^*|\perp_{BJ}^s|y^*|$ and so $x \perp_{BJ}^s y$.
\end{proof}

%%%%%%%%%%%%%%%%%%%%%%%%%%%%%%%%%%%%%%%%%%%%%%%%%%%%%%%%%
Let $\A$ be a $C^*$-algebra. Define the incoming and outgoing neighborhood of $x\in\A$ by
$$L_x^s = \{ y \in \A :y\perps  x\}\quad\hbox{  and }\quad R_x^s = \{ y \in \A : x\perps  y\}.$$   Define an equivalence relation on $\A$ by declaring that $x \sim y$ if $L_x^s = L_y^s$ and $R_x^s =R_y^s$.
%%%%%%%%%%%%%%%%%%%%%%%%%%%%%%%%%%%%%%%%%%%%%%%%%%%%%%%%%
\begin{corollary}\label{pure-Rset}

Let $\A$ be a  $C^*$-algebra, and let $x,y \in \A \setminus \{0\}$. If $\cP_{x^{*}} \subset \cP_{y^{*}}$, then $R_x^s \subset R_y^s$.
\begin{proof}
Suppose that $z \in R_x^s$. Then,  there exists a $\rho \in \cP_{|x^*|}=\cP_{x^{*}}$ such that $\rho (|z^*|)=0$. Since $\cP_{x^{*}} \subset \cP_{y^\ast}$, it follows that $y\perps  z$. Hence, $R_x^s \subset R_y^s$.
\end{proof}
\end{corollary}
%%%%%%%%%%%%%%%%%%%%%%%%%%%%%%%%%%%%%%%%%%%%%%%%%%%%%%%%%

%%%%%%%%%%%%%%%%%%%%%%%%%%%%%%%%%%%%%%%%%%%%%%%%%%%%%%%%%
\begin{lemma}\label{Rset-cpt}

Let $\cH$ be a Hilbert space, let $(P_\lambda )_{\lambda \in \Lambda} \subset \cB (\cH)$ be an orthogonal family of projections with sum $I$, and let $\A = \{ P_\lambda : \lambda \in \Lambda \}' \cap \cK (\cH)$. Suppose that $A,B \in \A \setminus \{0\}$. Then, the following are equivalent:
\begin{itemize}
\item[{\rm (i)}] $R_A^s \subset R_B^s$.
\item[{\rm (ii)}] $M_{|A^*|} \subset M_{|B^*|}$.
\item[{\rm (iii)}] $\cP_{|A^*|} \subset \cP_{|B^*|}$.
\end{itemize}
\begin{proof}
By Lemma~\ref{abs-polar}
   $R_A^s = R_{|A^*|}^s$ and $R_B^s = R_{|B^*|}^s$, so it may be assumed that $A=|A^*|\geq 0$ and $B =|B^*|\geq 0$. We may also assume that $\|A\|=\|B\|=1$.

(i) $\Rightarrow$ (ii).
Suppose that $M_A\not \subset M_B$. Let $\xi$ be a unit vector in $M_A \setminus M_B$, and let $\Lambda_\xi = \{ \lambda \in \Lambda : P_\lambda \xi \neq 0\}$. It follows from
\[
\xi = A\xi = \sum_{\lambda \in \Lambda}AP_\lambda \xi = \sum_{\lambda \in \Lambda_\xi}AP_\lambda \xi
\]
that
\[
1 = \sum_{\lambda \in \Lambda_\xi}\|AP_\lambda \xi\|^2 \leq \sum_{\lambda \in \Lambda_\xi} \|P_\lambda \xi\|^2 \leq \|\xi\|^2 = 1 ,
\]
which implies that $\|AP_\lambda \xi\|=\|P_\lambda \xi\|$ for each $\lambda \in \Lambda_\xi$, that is, $P_\lambda \xi \in M_A$ for each $\lambda \in \Lambda_\xi$. Meanwhile, since $\xi \not \in M_B$, we obtain
\[
1>\|B\xi\|^2 = \sum_{\lambda \in \Lambda} \|BP_\lambda \xi\|^2 = \sum_{\lambda \in \Lambda_\xi}\|BP_\lambda \xi\|^2 .
\]
From this, we infer that $\|BP_\lambda \xi\| < \|P_\lambda \xi\|$ for some $\lambda \in \Lambda_\xi$; otherwise, $BP_\lambda \xi=P_\lambda \xi$ for each $\lambda \in \Lambda_\xi$
 and
\[
B\xi = \sum_{\lambda \in \Lambda}BP_\lambda\xi = \sum_{\lambda \in \Lambda_\xi}BP_\lambda \xi = \sum_{\lambda \in \Lambda_\xi}P_\lambda\xi = \sum_{\lambda \in \Lambda}P_\lambda\xi =\xi ,
\]
which contradicts $\xi \not \in M_B$. Hence, $P_{\lambda_0} \xi \in M_A \setminus M_B$ for some $\lambda_0 \in \Lambda_\xi$. Replacing $\xi$ with $\|P_{\lambda_0} \xi\|^{-1}P_{\lambda_0} \xi$ if necessary, we may assume that $P_{\lambda_0} \xi =\xi$.

Now, set $\xi = \zeta_{\mu_0}$, and let $(\zeta_\mu)_{\mu \in M}$ be an orthonormal basis for $\cH$ such that $\zeta_\mu \in \bigcup_{\lambda \in \Lambda}P_\lambda(\cH)$ for each $\mu \in M$. Then, the net of projections $(E_J)_J$ directed by all finite subsets $J$ of $M$ with respect to inclusion relation forms an (increasing) approximate unit for $\cK (\cH)$, where $E_J$ is the projection onto $\langle \{ \zeta_\mu :\mu \in J\}\rangle$. Here, we note that $E_J = \sum_{\mu \in J}E_{\zeta_\mu} \in \A$ for each finite subset $J$ of $M$. Let $J$ be a finite subset of $M$ such that $\mu_0 \in J$ and $\|B-E_JB\|< 1-\|B\xi\|$,
 where $\|B\xi\|^2 = \langle B^2\xi,\xi \rangle \leq \langle B\xi,\xi \rangle <1$ by $\xi \not \in M_B$. Since $\langle A\xi,\xi \rangle = 1 =\|A\|$, $E_\xi=E_{\{\mu_0\}} \leq E_J$ and $\langle (E_J-E_\xi )\xi,\xi \rangle = 0$, it follows by Lemma~\ref{sBJ-pure} that $A\perps  E_J-E_\xi$. Meanwhile, we note that $E_\xi B^2E_\xi = \omega_\xi (B^2)E_\xi = \|B\xi\|^2E_\xi$, which implies that
\[
\|E_\xi B\| = \|BE_\xi \| = \|E_\xi B^2E_\xi \|^{1/2} = \|B\xi\| .
\]
Therefore, $B \not\perps  E_J-E_\xi$ by
\[
\|B-(E_J-E_\xi )B\| \leq \|B-E_JB\|+\|E_\xi B\|<1=\|B\| .
\]
Hence, $E_J -E_\xi \in R_A^s \setminus R_B^s$ which  proves (i) $\Rightarrow$ (ii).

(ii) $\Rightarrow$ (iii). Let $\omega_\xi \in \cP_A$. Then, $\langle A\xi,\xi \rangle =1$, which implies that $A\xi = \xi$. Hence, $\xi \in M_A \subset M_B$, in which case, $B\xi = \xi$ and $\omega_\xi (B) = 1$. This shows that $\omega_\xi \in \cP_B$. Thus, (ii) $\Rightarrow$ (iii).

The implication (iii) $\Rightarrow$ (i) follows from Corollary~\ref{pure-Rset}.
\end{proof}
\end{lemma}
%%%%%%%%%%%%%%%%%%%%%%%%%%%%%%%%%%%%%%%%%%%%%%%%%%%%%%%%%
Given a nonzero $x\in\A$, let $$\cN_x = \{ \rho \in \cP(\A): \rho (|x|)=0\}.$$

%%%%%%%%%%%%%%%%%%%%%%%%%%%%%%%%%%%%%%%%%%%%%%%%%%%%%%%%%
\begin{lemma}\label{Lset-pure}

 Let $\A$ be a  $C^*$-algebra, and let $x,y \in \A \setminus \{0\}$. If $\cN_{|x^*|} \subset \cN_{|y^*|}$, then $L_x^s \subset L_y^s$.
\begin{proof}
Let $z \in L_x^s \setminus \{0\}$. Then, $\rho (|x^*|)=0$ for some $\rho \in \cP_{|z^*|}$. Since $\rho \in \cN_{|x^*|} \subset \cN_{|y^*|}$, it follows that $z \in L_y^s$.
\end{proof}
\end{lemma}
%%%%%%%%%%%%%%%%%%%%%%%%%%%%%%%%%%%%%%%%%%%%%%%%%%%%%%%%%

%%%%%%%%%%%%%%%%%%%%%%%%%%%%%%%%%%%%%%%%%%%%%%%%%%%%%%%%%
\begin{lemma}\label{Lset-cpt}

Let $\cH$ be a Hilbert space, let $(P_\lambda )_{\lambda \in \Lambda} \subset \cB (\cH)$ be an orthogonal family of projections with sum $I$, and let $\A = \{ P_\lambda : \lambda \in \Lambda \}' \cap \cK (\cH)$. Suppose that $A,B \in \A \setminus \{0\}$. Then, the following are equivalent:
\begin{itemize}
\item[{\rm (i)}] $L_A^s \subset L_B^s$.
\item[{\rm (ii)}] $\ker |A^*| \subset \ker |B^*|$.
\item[{\rm (iii)}] $\cN_{|A^*|} \subset \cN_{|B^*|}$.
\end{itemize}
\begin{proof}
As in the proof of Lemma~\ref{Rset-cpt}, it may be assumed that $A=|A^*| \geq 0$, $B =|B^*|\geq 0$ and $\|A\|=\|B\|=1$. Suppose that $\ker A \not \subset \ker B$. Let $\xi$ be a unit vector in $\cH$ such that $\xi \in \ker A \setminus \ker B$. Since $A\xi = \sum_{\lambda \in \Lambda}AP_\lambda \xi$ and $B\xi = \sum_{\lambda \in \Lambda}BP_\lambda\xi$, we infer that $AP_\lambda \xi=0$ for each $\lambda \in \Lambda$ but $BP_{\lambda_0} \neq 0$ for some $\lambda_0 \in \Lambda$. Replacing $\xi$ with $\|P_{\lambda_0}\xi\|^{-1}P_{\lambda_0}\xi$ if necessary, we may assume that $P_{\lambda_0}\xi =\xi$, in which case, $E_\xi \in \A$.  By Lemma~\ref{sBJ-pure},
 $E_\xi\perps  A$. Meanwhile, if $\omega_\zeta \in\cP_{E_{\xi}}$, then $\langle E_\xi \zeta,\zeta\rangle =1$, which implies that $\zeta = E_\xi \zeta = \langle \zeta ,\xi\rangle \xi$ and $|\langle \zeta ,\xi \rangle |=1$. Hence,
\[
\omega_\zeta (B) = \omega_\xi (B) \geq \omega_\xi (B^2) = \|B\xi\|^2 >0
\]
by $\xi \not \in \ker B$. This means $E_\xi \not\perps  B$, that is, $E_\xi \in L_A^s \setminus L_B^s$. Therefore, (i) $\Rightarrow$ (ii).

Next, suppose that $\ker A \subset \ker B$. Let $\xi$ be a unit vector in $\cH$ such that $\omega_\xi \in \cN_A$. Then, we obtain $\|A^{1/2}\xi\|^2 = \omega_\xi (A) = 0$ and $A\xi = 0$, which implies that $\xi \in \ker A \subset \ker B$. It follows that $\omega_\xi (B) = 0$. Thus, $\cN_A \subset \cN_B$, and (ii) $\Rightarrow$ (iii).

The implication (iii) $\Rightarrow$ (i) follows from Lemma~\ref{Lset-pure}.
\end{proof}
\end{lemma}
%%%%%%%%%%%%%%%%%%%%%%%%%%%%%%%%%%%%%%%%%%%%%%%%%%%%%%%%%
\begin{remark}\label{rem:L_a-vs-Rng(A)}

Let $\cH$ be a Hilbert space, let $(P_\lambda )_{\lambda \in \Lambda} \subset \cB (\cH)$ be an orthogonal family of projections with sum $I$, and let $\A = \{ P_\lambda : \lambda \in \Lambda \}' \cap \cK (\cH)$. Suppose that $A,B \in \cK (\cH) \setminus \{0\}$. Then, $\ker A = A^*(\cH)^\perp$. Hence, by the preceding lemma, $L_A^s \subset L_B^s$ if and only if $\overline{|A^*|(\cH)} \supset \overline{|B^*|(\cH)}$ provided that $A,B \in \A$.
\end{remark}
%%%%%%%%%%%%%%%%%%%%%%%%%%%%%%%%%%%%%%%%%%%%%%%%%%%%%%%%%

%%%%%%%%%%%%%%%%%%%%%%%%%%%%%%%%%%%%%%%%%%%%%%%%%%%%%%%%%
\begin{lemma}\label{proj-cpt}

Let $\cH$ be a Hilbert space, let $(P_\lambda )_{\lambda \in \Lambda} \subset \cB (\cH)$ be an orthogonal family of projections with sum $I$, and let $\A = \{ P_\lambda : \lambda \in \Lambda \}' \cap \cK (\cH)$. Suppose that $A \in \A \setminus \{0\}$. Then, $\|A\|^{-1}|A^*|$ is a projection if and only if $B \in \A$ and $R_A^s = R_B^s$ imply $L_A^s \supset L_B^s$.
\begin{proof}
It may be assumed that $A =|A^\ast|\geq 0$ and $\|A\|=1$. Suppose that $A$ is a projection. In this case, we have $\overline{A(\cH)} =A(\cH) =M_A$. Let $B \in \A$ be such that $R_A^s = R_B^s$. Then, we obtain $B \neq 0$ and, by Lemma~\ref{Rset-cpt}, $\overline{A(\cH)} =M_A = M_{|B^*|} \subset \overline{|B^*|(\cH)}$, which implies that $L_A^s \supset L_B^s$.

Conversely, suppose that $B \in \A$ and $R_A^s = R_B^s$ imply $L_A^s \supset L_B^s$. We infer that $M_A$ is finite-dimensional by $A \in \cK (\cH)$. Let $E$ be the projection onto $M_A$. We note that $E =\chi_{\{1\}}(A) \in \A$ and, from $M_A = E(\cH) = M_E$ and Lemma~\ref{Rset-cpt}, that $R_A^s = R_E^s$. Hence, by the assumptions, $L_A^s \supset L_E^s$ which gives  $M_A \subset A(\cH) \subset \overline{A(\cH)} \subset E(\cH) = M_A$. Therefore, $M_A = A(\cH)$ and $A(A\xi) = A\xi$ for each $\xi \in \cH$. Combining this with the assumption that $A \geq 0$, we conclude that $A$ is a projection.
\end{proof}
\end{lemma}
%%%%%%%%%%%%%%%%%%%%%%%%%%%%%%%%%%%%%%%%%%%%%%%%%%%%%%%%%
For each pair $\xi,\zeta \in \cH$, let $\xi \otimes \zeta$ be a rank-one operator defined by $(\xi\otimes \zeta)\eta = \langle \eta,\zeta\rangle \xi$ for each $\eta \in \cH$. We note that $(\xi \otimes \zeta)^*=\zeta \otimes \xi$ and
\[
(\xi_1\otimes \zeta_1)(\xi_2\otimes\zeta_2) = \langle \zeta_1,\xi_2 \rangle (\xi_1\otimes \zeta_2) =\xi_1 \otimes ((\zeta_2 \otimes \zeta_1)\xi_2).
\]
%%%%%%%%%%%%%%%%%%%%%%%%%%%%%%%%%%%%%%%%%%%%%%%%%%%%%%%%%
\begin{remark}\label{rem:SVD-blockwise}

Let $\cH$ be a Hilbert space, let $(P_\lambda )_{\lambda \in \Lambda} \subset \cB (\cH)$ be an orthogonal family of projections with sum $I$, and let $\A = \{ P_\lambda : \lambda \in \Lambda \}' \cap \cK (\cH)$. Suppose that $\xi,\zeta \in \cH$. Then, $\xi \otimes \zeta \in \A$ if and only if $\xi ,\zeta \in P_\lambda (\cH)$ for some $\lambda \in \Lambda$. To show this, suppose that $\xi \otimes \zeta \in \A \setminus \{0\}$. Let $\lambda ,\mu \in \Lambda$ be such that $P_\lambda \xi \neq 0$ and $P_\mu \zeta \neq 0$. Since $\xi\otimes \zeta \in \A$ and
\[
0 \neq P_\lambda \xi \otimes P_\mu\zeta = P_\lambda (\xi \otimes \zeta )P_\mu = P_\lambda P_\mu (\xi \otimes \zeta) ,
\]
it follows that $\lambda = \mu$. Hence, $P_\lambda \xi = \xi$ and $P_\lambda \zeta =\zeta$ for a unique $\lambda \in \Lambda$. Conversely, if $\xi,\zeta \in P_\lambda (\cH)$ for some $\lambda \in \Lambda$, then $\xi \otimes \zeta =P_\lambda (\xi \otimes \zeta )P_\lambda \in  \cK (\cH_\lambda )\subseteq\A$.
\end{remark}
%%%%%%%%%%%%%%%%%%%%%%%%%%%%%%%%%%%%%%%%%%%%%%%%%%%%%%%%%
\begin{remark}

Let $\cH$ be a Hilbert space, let $(P_\lambda )_{\lambda \in \Lambda} \subset \cB (\cH)$ be an orthogonal family of projections with sum $I$, and let $\A = \{ P_\lambda : \lambda \in \Lambda \}' \cap \cK (\cH)$. Suppose that $E \in \A$ is a projection. Then, there exists an orthogonal family of rank-one projections $E_{\xi_1},\ldots ,E_{\xi_n} \in \A$ such that $E = \sum_{j=1}^n E_{\xi_n}$.   To show this, we note that $\rank E = n \in \mathbb{N}$ by $E \in \cK (\cH)$. Let $\Lambda_E = \{ \lambda \in \Lambda : EP_\lambda \neq 0\}$, and let $\lambda \in \Lambda_E$. Since $EP_\lambda =P_\lambda E$, it follows that $EP_\lambda$ is a nonzero projection. Moreover, if $\{\xi_{\lambda ,1},\ldots ,\xi_{\lambda ,n_\lambda}\}$ is an orthonormal basis for $P_\lambda E(\cH)$, then
\[
EP_\lambda =PE_\lambda = \sum_{j=1}^{n_\lambda} E_{\xi_{\lambda,j}}
\]
and $E_{\xi_{\lambda,j}} =\xi_{\lambda,j} \otimes \xi_{\lambda ,j} \in \A$ by $\xi_{\lambda,j} \in P_\lambda (\cH)$. Combining this with the finiteness of $\Lambda_E$, we obtain
\[
E = \sum_{\lambda \in \Lambda_E}EP_\lambda =\sum_{\lambda \in \Lambda_E}\sum_{j=1}^{n_\lambda}E_{\xi_{\lambda ,j}}
\]
as desired.
\end{remark}
%%%%%%%%%%%%%%%%%%%%%%%%%%%%%%%%%%%%%%%%%%%%%%%%%%%%%%%%%
The next lemma computes the rank of  operators on concrete compact $C^*$-algebras in terms on strong BJ orthogonality (it equals the maximal $n$ for which one can find the ascending sequence  of length~$n$, given in the statement of lemma).
\begin{lemma}\label{rank-n}

Let $\cH$ be a Hilbert space, let $(P_\lambda )_{\lambda \in \Lambda} \subset \cB (\cH)$ be an orthogonal family of projections with sum $I$, and let $\A = \{ P_\lambda : \lambda \in \Lambda \}' \cap \cK (\cH)$. Suppose that $A \in \A \setminus \{0\}$ and  $n \geq 1$. Then, $\rank A \geq n$ if and only if there exist $A_0,\ldots ,A_{n-1} \in \A \setminus \{0\}$, with $A_0=
A$, such that $L_{A_0}^s \subsetneq L_{A_1}^s \subsetneq \cdots \subsetneq L_{A_{n-1}}^s$.
\begin{proof}
Since $\overline{A(\cH)}=\overline{|A^*|(\cH)}$  and $L_A^s=L_{|A^*|}^s$, it may be assumed that $A \geq 0$. Suppose that $\rank A \geq n$. Considering the functional calculus for $A$, we infer that there exist a non-increasing sequence of nonnegative numbers $(\lambda_m)_m$ and an orthogonal sequence of rank-one projections $(E_m)_m$ in $\A$ such that $A = \sum_m \lambda_m E_m$, where the series converges in the norm topology. We note that $\lambda_n >0$ by $\rank A \geq n$. Let $A_k = \sum_{j=1}^{n-k} E_j \in \A \setminus \{0\}$ for $k \in \{1,\ldots ,n-1\}$. Then, we have
\[
\overline{A(\cH)} \supsetneq \overline{A_1(\cH)} \supsetneq \cdots \supsetneq \overline{A_{n-1}(\cH)} ,
\]
which, by Remark~\ref{rem:L_a-vs-Rng(A)}, implies that $L_A^s \subsetneq L_{A_1}^s \subsetneq \cdots \subsetneq L_{A_{n-1}}^s$.

Conversely, if $L_A^s \subsetneq L_{A_1}^s \subsetneq \cdots \subsetneq L_{A_{n-1}}^s$ for $A_1,\ldots ,A_{n-1} \in \A \setminus \{0\}$, then
\[
\overline{A(\cH)} \supsetneq \overline{A_1(\cH)} \supsetneq \cdots \supsetneq \overline{A_{n-1}(\cH)} \neq \{0\} .
\]
Therefore, $\dim\overline{A(\cH)} \geq n$, which ensures $\rank A \geq n$.
\end{proof}
\end{lemma}
%%%%%%%%%%%%%%%%%%%%%%%%%%%%%%%%%%%%%%%%%%%%%%%%%%%%%%%%%

%%%%%%%%%%%%%%%%%%%%%%%%%%%%%%%%%%%%%%%%%%%%%%%%%%%%%%%%%
\begin{lemma}\label{central-ortho}

Let $\cH$ be a Hilbert space, let $(P_\lambda )_{\lambda \in \Lambda} \subset \cB (\cH)$ be an orthogonal family of projections with sum $I$, let $\A = \{ P_\lambda : \lambda \in \Lambda \}' \cap \cK (\cH)$, and $$\Lambda_A = \{ \lambda \in \Lambda : AP_\lambda \neq 0\}$$ for each $A \in \A$. Suppose that $A,B \in \A \setminus \{0\}$. Then, $\Lambda_A \cap\Lambda_B \neq \emptyset$ if and only if there exists a rank-one operator $C \in \A$ such that $C \not\perps  A$ and $C \not\perps  B$.
\begin{proof}
It may be assumed that $A \geq 0$ and $B \geq 0$. Suppose that $\Lambda_A \cap \Lambda_B \neq \emptyset$. Let $\lambda \in \Lambda_A \cap \Lambda_B$, and let $\xi$ be a unit vector in $P_\lambda (\cH)$ such that $A\xi \neq 0$. If $B\xi \neq 0$, then we have $E_\xi \in \A$, $E_\xi \not\perps  A$ and $E_\xi \not\perps  B$ by $\cP_{E_\xi} = \{ \omega_\xi\}$, $\omega_\xi (A) = \|A^{1/2}\xi\|^2>0$ and $\omega_\xi (B) =\|B^{1/2}\xi\|^2>0$ (c.f., Lemma~\ref{sBJ-pure}). If $B\xi=0$, then $BE_\xi=0$ and $BP_\lambda (I-E_\xi) = B(P_\lambda-E_\xi) = BP_\lambda \neq 0$. In this case, let $\zeta$ be a unit vector in $P_\lambda (I-E_\xi )(\cH) = (I-E_\xi)P_\lambda (\cH)$ such that $B\zeta \neq 0$. Replacing $\zeta$ with $-\zeta$ if necessary, we may assume that $\re \langle A\xi,\zeta\rangle \geq 0$. Set $\eta = 2^{-1/2}(\xi+\zeta)$. Since
\[
\omega_\eta (A) = \langle A\eta ,\eta \rangle = 2^{-1}(\langle A\xi,\xi\rangle + 2\re \langle A\xi,\zeta \rangle +\langle A\zeta,\zeta \rangle) >0
\]
and
\[
\omega_\eta (B) = \langle B \eta,\eta \rangle =2^{-1}\langle B\zeta ,\zeta \rangle >0 ,
\]
it follows that $E_\eta \not\perps  A$ and $E_\eta \not\perps B$.

For the converse, suppose that $\Lambda_A \cap \Lambda_B = \emptyset$, and that $C$ is a rank-one element of $\A$. It may be assumed that $\|C\|=1$. Let $\xi,\zeta$ be unit vectors in $\cH$ such that $C=\xi \otimes \zeta$. Recall  that $P_\lambda \xi=\xi$ and $P_\lambda \zeta =\zeta$ for some $\lambda \in \Lambda$. If $C \not\perps  A$, then $|C^*| \not\perps  A$ and $\omega_\xi (A) \neq 0$ by $|C^*|=E_\xi$, which implies that $AP_\lambda \xi =A\xi \neq 0$. Hence, $\lambda \in \Lambda_A$. Combining this with $\Lambda_A \cap \Lambda_B=\emptyset$, we obtain $BP_\lambda =0$ and $B\xi = BP_\lambda \xi=0$, that is, $|C^*|\perps  B$. This means that there is no rank-one operator $C \in \A$ such that $C \not\perps  A$ and $C \not\perps  B$.
\end{proof}
\end{lemma}
%%%%%%%%%%%%%%%%%%%%%%%%%%%%%%%%%%%%%%%%%%%%%%%%%%%%%%%%%
\subsection{Characterizing compact $C^*$-algebras}

Let $\A$ be a $C^*$-algebra, and let $\mathcal{C} \subset \A \setminus \{0\}$. Then,  $\mathcal{C}$  is called an \emph{$R$-chain} if $R_x^s \subsetneq R_y^s$ or $R_x^s \supsetneq R_y^s$ whenever $x,y \in \mathcal{C}$ and $x \neq y$. Define the set $\mathfrak{F}_x$ by
\[
\mathfrak{F}_x = \{ \mathcal{C} \subset \A \setminus \{0\} : \text{$\mathcal{C}$ is an $R$-chain with $R_y^s \subset R_x^s$ for each $y \in \mathcal{C}$} \} .
\]
It is important to know if $\mathfrak{F}_x$ contains an infinite set or not.
%%%%%%%%%%%%%%%%%%%%%%%%%%%%%%%%%%%%%%%%%%%%%%%%%%%%%%%%%%
\begin{lemma}\label{R-chain}

Let $\A$ be a $C^*$-algebra, let $x \in \A$ be positive, and let $(x_n)_n$ be a sequence of positive elements of $\A$. Suppose that $\|x\|=\|x_n\|$ and $xx_n=x_n$ for each $n$, and that $x_mx_n = 0$ whenever $m \neq n$. Then, $\{ y_n : n \in \mathbb{N}\} \in \mathfrak{F}_x$, where $y_n = \sum_{j=1}^n x_j$.
\begin{proof}
It may be assumed that $\|x\|=1$. We note that $\|y_n\|=1$ for each $n$.
 Let $z \in R_{y_n}^s$. It follows from $xy_n=y_n$ that
\[
1 =\|y_n\| \leq \|y_n+zwy_n \| = \|xy_n+zwy_n\| \leq \|x+zw\|
 \]
for each $w \in \A$, which implies that $z \in R_x^s$. Hence, $R_{y_n}^s \subset R_x^s$ for each $n$. Meanwhile, since $z \in R_{y_n}^s$, Lemma~\ref{sBJ-pure} generates a pure state $\rho$ of $\A$ such that $\rho (y_n)=1$ and $\rho (|z^*|)=0$. It follows from  Lemma~\ref{lem:reduction} that $\rho(y_n^2)=\|y_n\|\rho(y_n)=1$.  and $z \in R_{y_n^2}^s$ again by Lemma~\ref{sBJ-pure}. Combining this with $y_n^2=y_{n+1}y_n$, we have
\[
1 = \|y_n^2\| \leq \|y_n^2+zwy_n\| = \|y_{n+1}y_n+zwy_n\| \leq \|y_{n+1}+zw\|
\]
for each $w\in \A$, that is, $z \in R_{y_{n+1}}^s$. Finally, we note that
\[
\|y_{n+1}+y_nw\| \geq \|x_{n+1}(y_{n+1}+y_nw)\| = \|x_{n+1}^2\| =1=\|y_{n+1}\|
\]
for each $w \in \A$. This shows that $y_n \in R_{y_{n+1}}^s \setminus R_{y_n}^s$. Therefore, $R_{y_n}^s \subsetneq R_{y_{n+1}}^s$ for each $n$, and $\{ y_n : n \in \mathbb{N}\} \in \mathfrak{F}_x$.
\end{proof}
\end{lemma}
%%%%%%%%%%%%%%%%%%%%%%%%%%%%%%%%%%%%%%%%%%%%%%%%%%%%%%%%%%
\begin{theorem}\label{cpt-char}

Let $\A$ be a $C^*$-algebra. Then, $\A$ is a compact $C^*$-algebra if and only if it satisfies the following conditions for each $x \in \A \setminus \{0\}$:
\begin{itemize}
\item[(i)] There exists a nonzero projection $e \in \A$ such that $e\A e=\mathbb{C}e$ and $R_e^s \subset R_x^s$.
\item[(ii)] Every element of $\mathfrak{F}_x$ is a finite set.
\end{itemize}
\begin{proof}
Suppose that $\A$ is a compact $C^*$-algebra. In this case, it may be assumed that $\A$ acts on a Hilbert space $\cH$ and there exists an orthogonal family of projections $(P_\lambda )_{\lambda \in \Lambda} \subset \cB (\cH)$ with sum $I$ such that $\A = \{ P_\lambda : \lambda \in \Lambda \}' \cap \cK (\cH)$. Let $A \in \A \setminus \{0\}$, and let $\lambda \in \Lambda$ be such that $\|AP_\lambda \|=\|A\|$. Since $|A^*|P_\lambda$ can be viewed as a compact operator on $P_\lambda (\cH)$, we have a unit vector $\xi \in P_\lambda (\cH)$ such that $|A^*| \xi = \||A^*|P_\lambda\|\xi = \|A\|\xi$. It follows from $E_\xi =\xi \otimes \xi \in \A$ and $M_{E_\xi} = \mathbb{C}\xi \subset M_{|A^*|}$ that $R_{E_\xi}^s \subset R_{|A^*|}^s = R_A^s$ by  Lemma~\ref{Rset-cpt}. Also, $E_{\xi}$ is a rank-one projection and hence $E_\xi\A E_{\xi}=\CC E_{\xi}$.

Next, let $\mathcal{C}$ be an element of $\mathfrak{F}_{A}$.  By Lemma~\ref{Rset-cpt}, we have $M_{|B^*|} \subset M_{|A^*|}$ for each $B \in \mathcal{C}$, while $M_{|A^*|}$ is finite-dimensional. Set $\dim M_{|A^*|} = n$. We claim that $\mathcal{C}$ contains at most $n$ elements. Indeed, if $\mathcal{C}$ contains mutually distinct elements $B_1,\ldots ,B_{n+1}$, then $\dim M_{|B_i^*|} = \dim M_{|B_j^*|}$ for some $i,j$ with $i \neq j$. However, in this case, we have neither $R_{B_i}^s \subsetneq R_{B_j}^s$ nor $R_{B_i}^s \supsetneq R_{B_j}^s$, which contradicts the choice of $\mathcal{C}$. This proves the claim.

Conversely, we assume that $\A$ satisfies the conditions (i) and (ii). It may be assumed that $\A$ acts on a Hilbert space $\cH$. Let $A \in \A \setminus \{0\}$. By (i), we have $R_E^s \subset R_A^s = R_{|A^*|}^s$ for some nonzero projection $E \in \A$ such that $E\A E = \mathbb{C}E$. Then, it follows from \cite[Theorem 2.5(b)]{AR14} that $(I-E)A \in R_E^s \subset R_A^s$, which implies that
\[
\|A\| \geq \|EA\| = \|A-(I-E)A\| =\lim_{\lambda}\|A-(I-E)AU_\lambda\| \geq \|A\| .
 \]
where $(U_{\lambda})_{\lambda}\in\A$ is an approximate identity.
 Hence,  $\|EA(EA)^*\|=\|EA\|^2=\|A\|^2$, so $EAA^*E = \|A\|^2E$ and $\|A^*\xi\|=\||A^*|\xi\|=\|A\|=\|A^*\|$ whenever $\xi \in E(\cH)$ and $\|\xi\|=1$. In particular, we obtain
$|A^*|E =\|A\|E$.

Let $A \in \A \setminus \{0\}$ be positive and $\|A\|=1$. By the preceding paragraph, there exists a nonzero projection $E \in \A$ such that \begin{equation}\label{eq:AE=E}
E\A E=\mathbb{C}E \quad\hbox{ and }\quad AE=E.
\end{equation}
We claim that there exists a finite family   of pairwise orthogonal nonzero projections $E_1,\ldots ,E_n \in \A$ such that $E_j\A E_j =\mathbb{C}E_j$ and $AE_j =E_j$ for each $j$ and $\|(I-\sum_{j=1}^nE_j)A\|<\|A\|$. To show this, suppose to the contrary that there is no such a family. Let $E_1,\ldots ,E_n \in \A$ be orthogonal family of projections such that $E_j\A E_j =\mathbb{C}E_j$ and $AE_j =E_j$ for each $j$. Since $\|(I-\sum_{j=1}^nE_j)A\|=\|A\|=1$ and $AE_j=E_j =E_jA$, it follows that
\[
\left( I-\sum_{j=1}^n E_j \right) A = A^{1/2}\left( I-\sum_{j=1}^n E_j \right)A^{1/2} \geq 0 ,
\]
which, by condition~(i) and identity~\eqref{eq:AE=E}, ensures the existence of a nonzero projection $E_{n+1} \in \A$ such that $E_{n+1}\A E_{n+1}=\mathbb{C}E_{n+1}$ and $(I-\sum_{j=1}^nE_j)AE_{n+1}=E_{n+1}$.  In particular, we derive
\[
E_iE_{n+1} = E_i\left( I-\sum_{j=1}^n E_j \right) AE_{n+1} = 0
\]
for each $i \in \{1,\ldots ,n\}$, that is, $\{E_1,\ldots, E_n,E_{n+1}\}$ is an orthogonal family of projections. As such, $E_{n+1}A=E_{n+1}\cdot(I-\sum_{j=1}^n E_j)A=((I-\sum_{j=1}^n E_j)A\cdot E_{n+1})^*=E_{n+1}$ Therefore, by an induction, we can construct an orthogonal sequence of nonzero projections $(E_n)_n$ such that $E_n\A E_n = \mathbb{C}E_n$ and $AE_n = E_n$ for each $n$. However, then $\{ \sum_{j=1}^n E_j : n \in \mathbb{N}\} \in \mathfrak{F}_A$ by Lemma~\ref{R-chain}, which contradicts (ii). Thus, we must have a finite orthogonal family of nonzero projections $E_1,\ldots ,E_n \in \A$ such that $E_j\A E_j =\mathbb{C}E_j$ and $AE_j =E_j$ for each $j$ and $\|(I-\sum_{j=1}^nE_j)A\|<\|A\|$.

Let $\A_0$ be the abelian $C^*$-subalgebra of $\A$ generated by $\{ A,E_1,\ldots ,E_n\}$. It may be assumed that $\A_0 = C_0(K)$ for some locally compact Hausdorff space $K$. Let $f$ and $e_j$ be the elements of $C_0(K)$ corresponding to $A$ and $E_j$. Since $C_0(K)e_j=\mathbb{C}e_j$, by Urysohn's lemma, we infer that $e_j$ is the characteristic function for a singleton $\{t_j\}$. If $1$ is not isolated in $\Sp (f)=\overline{f(K)}$, then there exists a net $(t_a)_a \subset K$ such that $f(t_a) \neq 1$ for each $a$ and $\lim_af(t_a) =1$. Since $\|f-\sum_{j=1}^ne_j\|<1$, we have $\|f-\sum_{j=1}^ne_j\|<|f(t_a)|$ for some index $a$. However, then, we have $t_a \in \{t_1,\ldots ,t_n\}$ and $f(t_a)=1$, which contradicts the choice of $(t_a)_a$. Hence, $1$ is isolated in $\Sp (f)$.

We have proved the following for each positive nonzero element $A \in \A$:
\begin{itemize}
\item[(a)] There exists a finite orthogonal family of nonzero projections $E_1,\ldots ,E_n \in \A$ such that $E_j \A E_j = \mathbb{C}E_j$ and $AE_j =\|A\|E_j$ for each $j$ and $\|(I-\sum_{j=1}^nE_j)A\|<\|A\|$.
\item[(b)] $\|A\|$ is isolated in $\Sp (A)$.
\end{itemize}
Let $A$ be a nonzero positive element of $\A$. Suppose that there exists a nonzero accumulation point $t \in \Sp (A)$. Then, there exists a strictly monotone sequence $(t_n)_n \subset (0,\|A\|)$ that converges to $t$. If $(t_n)_n$ is strictly increasing, then we choose an element $f_n \in C(\Sp (A))$ such that $0 \leq f_n \leq \1$, $f_n(t_n)=1$ and
\[
\supp f_n \subset \left( \frac{t_{n-1}+t_n}{2},\frac{t_n+t_{n+1}}{2}\right) ,
\]
where $t_0=0$. We note that $f_n(0)=0$ for each $n$ and $f_mf_n=0$ whenever $m \neq n$. If $(t_n)_n$ is strictly decreasing, the similar construction applies. We can also construct  a function $f\in C(\Sp (A))$ such that $0 \leq f \leq \1$, $f(0)=0$ and $\bigcup_n \supp f_n \subset f^{-1}(\{1\})$, in which case, $\|f\|=\|f_n\|=1$ and $ff_n = f_n$ for each $n$. However, Lemma~\ref{R-chain} prevents this construction because $\{ \sum_{j=1}^n f_n(A) : n \in \mathbb{N}\} \in \mathfrak{F}_{f(A)}$ which contradicts (ii). Therefore, $\Sp (A)$ does not contain any nonzero accumulation point. From this, $\A$ coincides with the closed linear span of its projections.

Finally, let  $\pi_{\rm ra} = \sum_\lambda\oplus  \pi_\lambda \colon   \A \to \cB(\bigoplus_\lambda  \cH_\lambda )$ (here, following~\cite{KR97b}, $\sum_\lambda\oplus $ denotes the direct $\ell_\infty$-sum as opposed to $\bigoplus$ for direct $c_0$-sums) be the reduced atomic representation of $\A$. We know that $\pi_{\rm ra}$ is faithful, and that weak-operator closure of $\pi_{\rm ra}(\A)$ coincides with $\mathcal{R}=\sum_\lambda\oplus \cB(\cH_\lambda )$  (see \cite[Proposition~10.3.10]{KR97b}). If $E \in \A$ satisfies $E\A E=\mathbb{C}E$, then $F\mathcal{R}F = \mathbb{C}F$, where $F =\pi_{\rm ra}(E)$. This means that $F$ is a minimal projection in the von Neumann algebra $\mathcal{R}$. It follows that $F$ is rank-one. Meanwhile, by (a), each projection $E \in \A$ has an orthogonal family of subprojections $E_1,\ldots ,E_n \in \A$ such that $E_j\A E_j = \mathbb{C}E_j$ and $E=\sum_{j=1}^nE_j$. Hence, $\pi_{\rm ra}(E) = \sum_{j=1}^n \pi_{\rm ra}(E_j)$ is a rank-$n$ projection, which in turn implies every elements of $\pi_{\rm ra}(\A)$ can be approximated in norm by a linear combination of finite-rank projections in  $\cB (\bigoplus_\lambda \cH_\lambda )$. Hence, we obtain $\pi_{\rm ra}(\A) \subset \cK (\bigoplus_\lambda  \cH_\lambda)$.
\end{proof}
\end{theorem}
%%%%%%%%%%%%%%%%%%%%%%%%%%%%%%%%%%%%%%%%%%%%%%%%%%%%%%%%%%
\begin{remark}
  It follows from~\cite[Corollary 2.8]{AGKRTZ24}  that a nonzero $a\in\A$ is a left-symmetric element if and only if $e:=|a^*|/\|a\|$ is also left-symmetric, and by~\cite[Corollary 2.7]{AGKRTZ24} this is further equivalent that $e$ is  a projection which satisfies $e\A e=\CC e$. Notice that Lemma~\ref{abs-polar} implies $R_a^s=R_{|a^*|}^s=R_e^s$.    Clearly, being a nonzero  element is equivalent to $a\not\perps a$.  Thus,  Theorem~\ref{cpt-char} can equivalently  be stated as follows:

  A $C^*$-algebra $\A$ is a compact $C^*$-algebra if and only if the following two conditions hold  for each $x\in\A$ with $x\not\perps x$:
  \begin{itemize}
      \item[(i')] There exists a left-symmetric  $a\in \A$,  with  $a\not\perps a$, and $R^s_a\subseteq R^s_x$.
      \item[(ii)]  Every element of $\mathfrak{F}_x$ is a finite set.
  \end{itemize}
  This gives a classification of compact $C^\ast$-algebras completely in terms of strong BJ relation.
  \end{remark}

\begin{corollary}\label{pre-cpt}

Let $\A$ and $\B$ be $C^*$-algebras. Suppose that $\A$ is a compact $C^*$-algebra, and that there exists a strong BJ isomorphism   from $\A$ onto $\B$. Then, $\B$ is also a compact $C^*$-algebra.
\end{corollary}
%%%%%%%%%%%%%%%%%%%%%%%%%%%%%%%%%%%%%%%%%%%%%%%%%%%%%%%%%%

\subsection{A preserver problem}
Recall that, if $A$  and $B$ are rank-one operators
then
\begin{equation}\label{eq:rk-1}
    A\perps  B \ \hbox{ is equivalent to }\ |B^*||A^*|=0.
\end{equation}
%%%%%%%%%%%%%%%%%%%%%%%%%%%%%%%%%%%%%%%%%%%%%%%%%%%%%%%%%
\begin{theorem}\label{nonlinear-1}

Let $\cH$ be a Hilbert space, let $(P_\lambda )_{\lambda \in \Lambda} \subset \cB (\cH)$ be an orthogonal family of projections with sum $I$, let $\A = \{ P_\lambda : \lambda \in \Lambda \}' \cap \cK (\cH)$, and let $\B$ be a compact $C^*$-algebra. Suppose that $T\colon \A \to \B$ is a strong  BJ isomorphism.  Then, there exists a linear $*$-isomorphism $\Psi \colon \B \to \A$ such that $(\Psi \circ T)( \sum_{\lambda \in \Lambda_0}\A P_\lambda ) = \sum_{\lambda \in \Lambda_0}\A P_\lambda$ for each  $\Lambda_0 \subset \Lambda$.
\begin{proof}
By Corollary~\ref{pre-cpt} it may be assumed that $\B$ acts on a Hilbert space $\cH_\B$ and there exists an orthogonal family of projections $(Q_\mu)_{\mu \in M}$ with sum $I$ such that
$$\B = \{ Q_\mu : \mu \in M\}' \cap \cK (\cH_\B).$$ By Lemma~\ref{rank-n} and $T (0)=0$, we have $\rank T (A) = \rank A$ for each $A \in \A$. Let $(e_{\lambda ,j})_{j \in J_\lambda}$ be an orthonormal basis for $P_\lambda (\cH)$, let $E_{\lambda ,j}$ be the rank-one projection onto $\mathbb{C}e_{\lambda,j}$, and let $F_{\lambda,j} = \|T (E_{\lambda ,j})\|^{-1}|T (E_{\lambda ,j})^*|$. Then,
 by~\eqref{eq:rk-1} $\{ F_{\lambda ,j} : \lambda \in \Lambda ,~j \in J_\lambda \}$ is an orthogonal family of rank-one projections. Let $\mu_{\lambda ,j}$ be a unique element of $\{ \mu \in M  : F_{\lambda ,j}P_\mu \neq 0\}$ for each $\lambda \in \Lambda$ and each $j \in J_\lambda$. It follows from Lemma~\ref{central-ortho} that we have $\lambda = \lambda'$ if and only if
\[
\{ \mu_{\lambda ,j}\}=\{ \mu \in M : F_{\lambda ,j}P_\mu \neq 0\} = \{\mu \in M :F_{\lambda',j'}P_\mu \neq 0\} = \{\mu_{\lambda',j'}\}
\]
whenever $j \in J_\lambda$ and $j' \in J_{\lambda'}$. From this, we can define an injective mapping $\varphi \colon  \Lambda \to M$ by $\varphi (\lambda ) = \mu_{\lambda ,j}$ for each $\lambda \in \Lambda$.
 If there exists an index $\mu \in M \setminus \varphi (\Lambda)$, then a rank-one projection $F \leq Q_\mu$ satisfies $F\perps  F_{\lambda,j}$ for each $\lambda \in \Lambda$ and each $j \in J_\lambda$, which implies that $|T^{-1}(F)^*|\perps  E_{\lambda ,j}$ for each $\lambda \in \Lambda$ and each $j \in J_\lambda$. However, since $\sum_{\lambda \in \Lambda}\sum_{j \in J_\lambda}E_{\lambda ,j} = \sum_{\lambda \in \Lambda}P_\lambda =I$ and $|T^{-1}(F)^*|$ is a scalar multiple of a rank-one projection, it follows that $|T^{-1}(F)^*|=0$, a contradiction. Hence, $\varphi (\Lambda )=M$ and $\varphi$ is a bijection.

Let $A \in \A \setminus \{0\}$.  We claim that $\{ \mu \in M : T (A)Q_\mu \neq 0\} = \varphi (\Lambda_A)$. To show this, let $\lambda \in \Lambda$ and $j \in J_\lambda$. Then, $\lambda \in \Lambda_A$ if and only if $\Lambda_A \cap \Lambda_{E_{\lambda ,j}} \neq \emptyset$, which occurs if and only if
\[
\{ \mu \in M : T (A)Q_\mu \neq 0\} \cap \{ \mu \in M : T (E_{\lambda,j})Q_\mu \neq 0\}\neq \emptyset
\]
by Lemma~\ref{central-ortho}. This last statement is also equivalent to
\[
\varphi (\lambda) \in \{ \mu \in M : T (A)Q_\mu \neq 0\}
\]
by $|T (E_{\lambda,j})^*|Q_\mu = |(T (E_{\lambda ,j})Q_\mu)^*|$ for each $\mu \in M$. Therefore, $\{ \mu \in M : T (A)Q_\mu \neq 0\} = \varphi (\Lambda_A)$.

Fix an index $\lambda \in \Lambda$. If $F = \sum_{j \in J_\lambda}F_{\lambda ,j} <Q_{\varphi (\lambda)}$, then there exists a rank-one projection $F_1 \leq Q_{\varphi (\lambda)}-F$. We note that $F_1 \in \B$ and, by~\eqref{eq:rk-1}, $|T^{-1} (F_1)^*|\perps  E_{\lambda ,j}$ for each $j \in J_\lambda$. Since $|T^{-1} (F_1)^*|$ is a scalar multiple of a rank-one projection, it follows that $|T^{-1}(F_1)^*|P_\lambda = 0$ and $T^{-1}(F_1)P_\lambda =0$. However, this means that
\[
\varphi (\lambda ) \not \in \varphi (\Lambda_{T^{-1} (F_1)}) = \{ \mu \in M : F_1Q_\mu \neq 0\} = \{ \varphi (\lambda)\} ,
\]
a contradiction. Thus, $F=Q_{\varphi (\lambda)}$ and $\dim Q_{\varphi (\lambda)}(\cH_\B) = \card (J_\lambda) = \dim P_\lambda (\cH)$. From this, we have an isometric isomorphism $U_\lambda \colon   P_\lambda (\cH) \to Q_{\varphi (\lambda)}(\cH_\B)$. Set $V_\lambda = U_\lambda P_\lambda =Q_{\varphi (\lambda)}U_\lambda P_\lambda \in \cB(\cH,\cH_\B)$.

Now, we can define an isometric isomorphism $V = \sum_{\lambda \in \Lambda}V_\lambda$ from $\cH$ onto $\cH_\B$ in the sense of the strong-operator topology. Let $\Psi (B) = V^*BV$ for each $B \in \B$. Then, we obtain
\begin{align*}
V^*BVP_\lambda = V^*BQ_{\varphi (\lambda)}U_\lambda P_\lambda
&=V^*Q_{\varphi (\lambda)}BQ_{\varphi (\lambda)}U_\lambda P_\lambda \\
&= P_\lambda U_\lambda^*Q_{\varphi (\lambda)}BQ_{\varphi (\lambda)}U_\lambda P_\lambda
\end{align*}
and
\begin{align*}
P_\lambda V^*BV = (V^*B^*VP_\lambda)^*
&= (P_\lambda U_\lambda^*Q_{\varphi (\lambda)}B^*Q_{\varphi (\lambda)}U_\lambda P_\lambda)^* \\
&= P_\lambda U_\lambda^*Q_{\varphi (\lambda)}BQ_{\varphi (\lambda)}U_\lambda P_\lambda \\
&= V^*BVP_\lambda
\end{align*}
for each $B \in \B$ and each $\lambda \in \Lambda$. Combining this with $B \in \cK (\cH_\B)$, we infer that $\Psi (B) \in \A$. Meanwhile, if $A \in \A$, then
\[
VAV^* Q_{\varphi (\lambda)} = Q_{\varphi (\lambda)}U_\lambda P_\lambda AP_\lambda U_\lambda^*Q_{\varphi (\lambda)} = Q_{\varphi (\lambda)}VAV^*
\]
for each $\lambda \in \Lambda$, which together with $A \in \cK (\cH)$ implies that $VAV^* \in \B$ and $A = \Psi (VAV^*)$. This shows that $\Psi$ is a $*$-isomorphism from $\B$ onto $\A$.

Finally, let $B \in \B$ and $\mu \in M$. Set $\mu = \varphi (\lambda )$. Since
\begin{align*}
\|\Psi (B)P_\lambda \|=\|V^*BVP_\lambda \| = \|BVP_\lambda \|
&= \|BQ_{\varphi (\lambda)}U_\lambda P_\lambda \| \\
&= \|P_\lambda U_\lambda^* Q_{\varphi (\lambda )}B^*\| \\
&= \|U_\lambda ^*Q_{\varphi (\lambda )}B^*\| \\
&= \|Q_{\varphi (\lambda )}B^*\| \\
&= \|BQ_{\varphi (\lambda )}\| ,
\end{align*}
it follows that $\Lambda_{\Psi (B)} = \varphi^{-1}(\{ \mu \in M : BQ_\mu \neq 0\})$. Therefore,
\[
\Lambda_{(\Psi \circ T)(A)} = \varphi^{-1}(\{ \mu \in M : T (A)Q_\mu \neq 0\}) = \Lambda_A
\]
for each $A \in \A$. This completes the proof because $A \in \sum_{\lambda \in \Lambda_0}\A P_\lambda$ if and only if $\Lambda_A \subset \Lambda_0$.
\end{proof}
\end{theorem}
%%%%%%%%%%%%%%%%%%%%%%%%%%%%%%%%%%%%%%%%%%%%%%%%%%%%%%%%%
\begin{corollary}[\cite{AGKRTZ24}]\label{nonlinear-iso}

Let $\A$ and $\B$ be compact $C^*$-algebras. If $\A \sim_{BJ}^s\B$, then $\A$ and $\B$ are $*$-isomorphic.
\end{corollary}
%%%%%%%%%%%%%%%%%%%%%%%%%%%%%%%%%%%%%%%%%%%%%%%%%%%%%%%%%

\subsection{Reduction to blocks and wild nonlinear examples}

%%%%%%%%%%%%%%%%%%%%%%%%%%%%%%%%%%%%%%%%%%%%%%%%%%%%%%%%%
\begin{lemma}\label{reduction}

Let $\cH$ be a Hilbert space, let $(P_\lambda )_{\lambda \in \Lambda} \subset \cB (\cH)$ be an orthogonal family of projections with sum $I$, let $\A = \{ P_\lambda : \lambda \in \Lambda \}' \cap \cK (\cH)$. Suppose that $\Phi \colon  \A \to \A$ be a strong BJ isomorphism  and $\Phi ( \sum_{\lambda \in \Lambda_0}\A P_\lambda ) = \sum_{\lambda \in \Lambda_0}\A P_\lambda$ for each $\Lambda_0 \subset \Lambda$. Then, $\Phi|(\sum_{\lambda \in \Lambda_0}\A P_\lambda)$ is also a strong Birkhoff-James orthogonality preserver for each $\Lambda_0 \subset \Lambda$.
\begin{proof}
Let $\Lambda_0 \subset \Lambda$, and let $A,B \in \A$. Set $A_0 = \sum_{\lambda \in \Lambda_0}AP_\lambda$ and $B_0 = \sum_{\lambda \in \Lambda_0}BP_\lambda$. Then, $A_0\perps B_0$ in $\A$ if and only if $\|\|B_0\|^2A_0-B_0B_0^*A_0\|=\|A_0\|\|B_0\|^2$, which occurs if and only if $A_0\perps  B_0$ in $\sum_{\lambda \in \Lambda_0}\A P_\lambda$. This proves the lemma.
\end{proof}
\end{lemma}
%%%%%%%%%%%%%%%%%%%%%%%%%%%%%%%%%%%%%%%%%%%%%%%%%%%%%%%%%

%%%%%%%%%%%%%%%%%%%%%%%%%%%%%%%%%%%%%%%%%%%%%%%%%%%%%%%%%
Here is our main result on a $C^*$-algebra  of \emph{all compact operators}. Recall that a conjugate-linear bijection $U\colon \cH\to \cH$ such that $\|Ux\|=\|x\|$ for every $x$ is called an anti-unitary operator. A prototypical example is
$$J\colon x=\sum \alpha_i e_i\mapsto\bar{x}:= \sum \overline{\alpha}_i e_i,$$
i.e., a conjugation with respect to a fixed orthonormal basis $(e_i)_i$ of $\cH$. Notice that $J$ is an involutive isometry and $U':=UJ$ is a linear bijective isometry such that $U=U'J$.
\begin{theorem}\label{nonlinear-2}

Let $\cH$ be a Hilbert space with $\dim \cH \geq 3$, and let $\Phi \colon  \cK (\cH)\to \cK (\cH)$ be a bijection. Then, $\Phi$ is a strong BJ isomorphism
 if and only if there exist a unitary or anti-unitary operator $U$ on $\cH$ such that  $M_{|\Phi (A)^*|} = U(M_{|A^*|})$ and $\ker |\Phi (A)^*| = U(\ker |A^*|)$ for each $A \in \cK (\cH)$.
\begin{proof}
We assume that $\Phi$ is a strong Birkhoff-James orthogonality preserver (in both directions). Let $\Gamma_r=\Gamma_r (\cK (\cH))$  be the reduced strong ortho-digraph  of $\cK (\cH)$ introduced in \cite{Tan25}, with vertex set $V(\Gamma_r)=\{ A^\sim : A \in \cK (\cH)\}$ where
\[
A^\sim = \{ B \in \cK (\cH) : \text{$R_A^s = R_B^s$ and $L_A^s = L_B^s$}\}
\]
and declaring that  $A^\sim$ forms a directed edge to  $B^\sim $, in notation $A^\sim \to B^\sim$, if $A\perps  B$. We note that $0^\sim \to  0^\sim$ is the  only loop in $\Gamma_r (\cK (\cH))$. By \cite[Proposition 3.4]{Tan25}, the formula $f (A^\sim ) = \Phi (A)^\sim$ defines a graph automorphism on $\Gamma_r (\cK (\cH))$. Let
\begin{align*}
L_{A^\sim} &= \{ B^\sim \in \Gamma_r (\cK (\cH)) : B^\sim \to A^\sim \}\\
R_{A^\sim} &= \{ B^\sim \in \Gamma_r (\cK (\cH)) : A^\sim \to B^\sim \}
\end{align*}
for each $A,B \in \Gamma_r (\cK (\cH))$. Recall that \cite[Lemma 3.3]{Tan25} ensures the following for nonzero $A,B$:
\begin{itemize}
\item[(i)] $L_A^s \subset L_B^s$ if and only if $L_{A^\sim} \subset L_{B^\sim}$.
\item[(ii)] $R_A^s \subset R_B^s$ if and only if $R_{A^\sim} \subset R_{B^\sim}$.
\end{itemize}
Combining this with Lemmas~\ref{Lset-cpt} and \ref{proj-cpt}, we infer that $f$ preserves the equivalence classes of finite-rank projections on $\cH$ and the ranks of representative. Now, let $\mathbb{P}_1(\cH)$ be the family of rank-one projections on $\cH$, and let $g(E) = E^\sim$ for each $E \in \mathbb{P}_1(\cH)$. Since a rank-one projection $E$ is uniquely determined by its one-dimensional space $M_E$, Lemma~\ref{Rset-cpt} implies that each equivalence class $A^\sim$, $A\in\A$,  contains at most one rank-one projection. Then, $h =g^{-1} \circ f \circ g$ defines a bijection on $\mathbb{P}_1(\cH)$ such that $EF=0$ if and only if $h(E)h(F)=0$ (see~\eqref{eq:rk-1}). By the Uhlhorn theorem, we obtain a unitary or an anti-unitary operator on $\cH$ such that $h(E) = UEU^{-1}$ for each $E \in \mathbb{P}_1 (\cH)$. We note  (see Lemma~\ref{abs-polar}) that $A^\sim = |A^*|^\sim$ for each $A \in \cK (\cH)$, and that $f(E^\sim ) = |\Phi (E)^*|^\sim$ for each $E \in \mathbb{P}_1(\cH)$. It follows
\[
\|\Phi(E)\|^{-1}|\Phi (E)^*| = g^{-1}(|\Phi (E)^*|^\sim) = h (E) = UEU^{-1}
\]
for each $E \in \mathbb{P}_1 (\cH)$. In particular, $M_{|\Phi (E)^*|} = M_{UEU^{-1}} = U(E(\cH))$ and
\[
\ker |\Phi (E)^*| = \ker UEU^{-1} = \ker EU^{-1}=U(\ker E).
\]

Let $A \in \cK (\cH)$, let $\xi \in M_{|A^*|}$, and let $\zeta \in \ker |A^*|$. Then, $M_{E_\xi} = \mathbb{C}\xi \subset M_{|A^*|}$, which together with Lemma~\ref{Rset-cpt} shows that $U\xi \in U(E_\xi(\cH)) = M_{|\Phi (E_\xi)^*|} \subset M_{|\Phi (A)^*|}$. Meanwhile, since $E_\zeta\perps  |A^*|$, it follows that $\Phi (E_\zeta)\perps  \Phi (A)$. Combining this with $M_{\Phi (E_\zeta)} = U(E_\zeta (\cH)) =\mathbb{C}(U\zeta)$, Lemma~\ref{sBJ-pure} implies  $\omega_{U\zeta}(|\Phi (A)^*|)=0$ and $U\zeta \in \ker |\Phi (A)^*|$. Therefore, $U(M_{|A^*|}) \subset M_{|\Phi (A)^*|}$ and $U(\ker |A^*|) \subset \ker |\Phi(A)^*|$.

To show the reverse inclusion, we note that $\Phi^{-1}$ induces $f^{-1}$ which has the same properties as $f$, that $h^{-1}=g^{-1}\circ f^{-1} \circ g$, and that $h^{-1}(E) = U^{-1}EU$ for each $E \in \mathbb{P}_1(\cH)$. Thus, an argument exactly the same as that for $f$ proves that $U^{-1}(M_{|A^*|}) \subset M_{|\Phi^{-1}(A)^*|}$ and $U^{-1}(\ker |A^*|) \subset \ker |\Phi^{-1}(A)^*|$ for each $A \in \cK (\cH)$. Replacing $A$ with $\Phi (A)$ in these inclusion relations, we get $M_{|\Phi (A)^*|} = U(M_{|A^*|})$ and $\ker |\Phi (A)^*| = U(\ker |A^*|)$ for each $A \in \cK (\cH)$, as desired.

Conversely, suppose that $\Phi$ satisfies the latter condition, and that  $A,B \in \cK (\cH)\setminus \{0\}$. If $A\perps  B$, then, by Lemmas~\ref{cpt-pure-state} and \ref{sBJ-pure} there exists a unit vector $\xi \in \cH$ such that $\omega_\xi(|A^*|)=\|A\|$ and $\omega_\xi(|B^*|)=0$, or equivalently (using equality in Cauchy-Schwarz):
$$|A^*|\xi=\|A\|\xi\hbox{ and }\langle |B^*|\xi,\xi\rangle= 0.$$
Hence,
\[
\||B^*|^{1/2}\xi\|^2 = \langle |B^*|\xi,\xi\rangle  = 0
\]
and $|B^*|\xi=0$. Thus, $\xi\in M_{|A^*|}\cap\ker |A^*|$. It follows from the assumption that $U\xi \in M_{|\Phi (A)^*|} \cap \ker |\Phi (B)^*|$,  which implies that $|\Phi (A)^*| \perp_{BJ} |\Phi (B)^*|$ and $\Phi (A)\perps  \Phi (B)$. Meanwhile, since $M_{|\Phi^{-1}(A)^*|} = U^* (M_{|A^*|})$ and $\ker |\Phi^{-1} (A)^*| = U^* (\ker |A^*|)$ for each $A \in \cK (\cH)$, it also follows that $A\perps B$ whenever $\Phi (A)\perps  \Phi (B)$. Thus, $\Phi$ preserves strong Birkhoff-James orthogonality in both directions.
\end{proof}
\end{theorem}
\begin{remark}\label{rem:nonlinearBJpreservers-on-K(H)}
Clearly, if  $V$ is  unitary or anti-unitary on $\cH$, then  $VM_{|A^*|}=M_{|(VA)^*|}$ and $V\ker |A^*|=\ker |(VA)^*|$ for every $A\in\cK(\cH)$. Therefore,
Theorem~\ref{nonlinear-2} can be equivalently  formulated as follows: Suppose $\dim\cH\ge 3$. Then,  $\Phi \colon  \cK (\cH)\to \cK (\cH)$  is a strong BJ isomorphism if and only if there exists a unitary or anti-unitary $U$ on $\cH$ such that $\Psi:=U^{-1}\Phi$ satisfies
\begin{equation}\label{eq:PropertyP}
   M_{|\Psi (A)^*|} = M_{|A^*|}\quad \hbox{ and }\quad \ker |\Psi (A)^*| = \ker |A^*|.
\end{equation}
\end{remark}
%%%%%%%%%%%%%%%%%%%%%%%%%%%%%%%%%%%%%%%%%%%%%%%%%%%%%%%%%
Theorem~\ref{nonlinear-2} and Remark~\ref{rem:nonlinearBJpreservers-on-K(H)} allows us to construct many wild examples of nonlinear strong Birkhoff-James orthogonality preservers on $\cK (\cH)$.

%%%%%%%%%%%%%%%%%%%%%%%%%%%%%%%%%%%%%%%%%%%%%%%%%%%%%%%%%
\begin{remark}

Let $\cH$ be a Hilbert space, and let $(M,N)$ be a pair  of  closed subspaces of $\cH$. Then, there exists a positive operator $A \in \cK (\cH) \setminus \{0\}$ such that $M=M_A$ and $N = \ker A$ if and only if $M \neq \{0\}$ is finite-dimensional, $M \perp N$ and $N^\perp$ is separable. Indeed, we know that $M_A$ is finite-dimensional and $M_A = \{ \xi \in \cH : A\xi = \xi\}$ for every nonzero positive compact operator $A$ on $\cH$. Hence, $\langle \xi,\zeta \rangle =\langle A\xi,\zeta \rangle =\langle \xi,A\zeta \rangle =0$ whenever $\xi \in M_A$ and $\zeta \in \ker A$. Moreover, $N^\perp = \overline{A(\cH)}$ is separable.

Conversely, suppose that $M$ is finite-dimensional, $M \perp N$ and $N^\perp$ is separable. Let $(\xi_n)_n$ be an orthonormal basis for $N^\perp (\supset M)$ such that $\xi_1,\ldots ,\xi_m$ is a basis for $M$. Define an operator $A$ on $\cH$ by
\[
A\xi = \sum_{j=1}^m \langle \xi,\xi_j \rangle \xi_j + \sum_{n \geq m+1}2^{-n}\langle \xi,\xi_n\rangle\xi_n
\]
for each $\xi \in \cH$. It follows that $A \geq 0$, $M_A = M$ and $\overline{A(\cH)} = N^\perp$, that is, $\ker A = N$.
\end{remark}
%%%%%%%%%%%%%%%%%%%%%%%%%%%%%%%%%%%%%%%%%%%%%%%%%%%%%%%%%
\begin{example}\label{exa:wild}

Let $\cH$ be a Hilbert space with $\dim\cH\ge 1$ and let
\[
C(M,N) = \{ A \in \cK (\cH) : \text{$M=M_{|A^*|}$ and $N=\ker |A^*|$}\}
\]
for each pair of closed subspaces $(M,N)$ of $\cH$ such that $M (\neq \{0\})$ is finite-dimensional, $M \perp N$ and $N^\perp$ is separable. From the definition of $C(M,N)$, it turns out that $C(M_1,N_1) \cap C(M_2,N_2) = \emptyset$ unless $M_1=M_2$ and $N_1=N_2$. Now, let $f_{M,N}$ be any bijection from $C(M,N)$ onto $C(M,N)$, and let $\Phi (0)=0$ and $\Phi (A) = f_{M,N}(A)$ for each $A \in C(M,N)$. Since each $A \in \cK (\cH) \setminus \{0\}$ belongs to exactly one $C(M,N)$, the mapping $\Phi$ is a bijection on $\cK (\cH)$. Now, let $A \in C(M,N)$. Then, $M=M_{|A^*|}$, $N=\ker |A^*|$ and $\Phi (A) = f_{M,N}(A) \in C(M,N)$. It follows that $M_{|\Phi (A)^*|}=M_{|A^*|}$ and
\[
\ker |\Phi (A)^*| = \ker |A^*| .
\]
Thus, by Theorem~\ref{nonlinear-2}, $\Phi$ is preserves strong Birkhoff-James orthogonality in both directions.  Conversely, if $\dim\cH\ge 3$, then, as follows from Remark~\ref{rem:nonlinearBJpreservers-on-K(H)}, every strong BJ isomorphhism $\Phi\colon\cK(\cH)\to\cK(\cH)$ is of the above form when composed with a suitable unitary or anti-unitary $U$. We remark that, in particular, such $\Phi$ are typically  not linear at all in general.
\end{example}
%%%%%%%%%%%%%%%%%%%%%%%%%%%%%%%%%%%%%%%%%%%%%%%%%%%%%%%%%

\subsection{The linear case}

%%%%%%%%%%%%%%%%%%%%%%%%%%%%%%%%%%%%%%%%%%%%%%%%%%%%%%%%%

%%%%%%%%%%%%%%%%%%%%%%%%%%%%%%%%%%%%%%%%%%%%%%%%%%%%%%%%%

%%%%%%%%%%%%%%%%%%%%%%%%%%%%%%%%%%%%%%%%%%%%%%%%%%%%%%%%%
\begin{lemma}\label{rank-one-norm}

Let $\cH$ be a Hilbert space, let $(P_\lambda )_{\lambda \in \Lambda} \subset \cB (\cH)$ be an orthogonal family of projections with sum $I$, let $\A = \{ P_\lambda : \lambda \in \Lambda \}' \cap \cK (\cH)$. Suppose a linear $\Phi \colon  \A \to \A$ is a  strong BJ isomorphism  such that $\Phi ( \sum_{\lambda \in \Lambda_0}\A P_\lambda ) = \sum_{\lambda \in \Lambda_0}\A P_\lambda$ for each
 $\Lambda_0\subset\Lambda$. Then, there exists a positive number $\kappa$ such that $\|\Phi (A)\|=\kappa$ whenever $A \in \A$ is a rank-one operator with $\|A\|=1$.
\begin{proof}
Suppose that $\Lambda$ is not a singleton. Let $A$ and $B$ be rank-one operators with $\|A\|=\|B\|=1$. In this case, $|A^*|$ and $|B^*|$ are rank-one projections. If $\Lambda_A \neq \Lambda_B$, then $AB^* = |A^*||B^*|=0$, which implies that $|(A+B)^*| = |A^*|+|B^*|$ is a rank-two projection. By Lemmas~\ref{proj-cpt} and~\ref{rank-n}, $|(\Phi (A)+\Phi (B))^*|$ is a scalar multiple of a rank-two projection. Meanwhile, $|\Phi (A)^*|$ and $|\Phi (B)^*|$ are also rank-one by Lemma~\ref{rank-n}, and $\Lambda_{\Phi (A)} \neq \Lambda_{\Phi (B)}$ by the assumption  $\Phi ( \sum_{\lambda \in \Lambda_0}\A P_\lambda ) = \sum_{\lambda \in \Lambda_0}\A P_\lambda$.
 Hence, $|(\Phi (A)+\Phi (B))^*| = |\Phi (A)^*|+|\Phi (B)^*|$ and $|\Phi (A)^*||\Phi (B)^*|=0$, which ensures that $\|\Phi (A)\|=\|\Phi (B)\|$.

Now, fix an index $\lambda_0$ and a rank-one operator $A_0$ with $\|A_0\|=1$ and $\Lambda_{A_0} = \{ \lambda_0\}$. Let $A$ be a rank-one operator with $\|A\|=1$. If $\Lambda_A \neq \{\lambda_0\}$, then $\|\Phi (A)\|=\|\Phi (A_0)\|$ by the preceding paragraph. Suppose that $\Lambda_A = \{\lambda_0\}$. Since $\Lambda$ contains an element $\lambda$ other than $\lambda_0$, there exists a rank-one projection $E \in \A$ such that $\Lambda_E=\{\lambda\} \neq \{\lambda_0\}$. Again by the preceding paragraph, we derive $\|A\|=\|E\|=\|A_0\|$.

Next, suppose that $\Lambda$ is a singleton, that is, $\A = \cK (\cH)$. If $\dim\cH<\infty$, then $\A=\cK(\cH)=\cB(\cH)$ is a (finite-dimensional) von Neumann algebra and  the result follows from \cite[Theorem 5.7]{KST18}. In the sequel we assume $\dim\cH=\infty$.
 Take a pair of  unit vectors $\xi,\zeta \in \cH$ and let $E_\xi,E_\zeta$ be the associated rank-one projections. Suppose that $\xi \perp \zeta$. Then,  $E_\xi E_\zeta=0$, so
$E_\xi\perps E_\zeta$ and  $|E_\xi+E_\zeta|=|E_\xi|+|E_\zeta|$. By  Lemmas~\ref{proj-cpt}, \ref{rank-n}, and equivalence~\eqref{eq:rk-1}, $\Phi(E_\xi)$ and $\Phi(E_\zeta)$ are scalar multiples of orthogonal rank-one projections and their sum, $\Phi(E_\xi+E_\zeta)=\Phi(E_\xi)+\Phi(E_\zeta)$ is also a (rank-two) projection. As above we deduce that $\|\Phi(E_\xi)\|=\|\Phi(E_\zeta)\|$.

Now, we omit the assumption that $\xi \perp \zeta$. Since $\dim \cH \geq 3$, there exists a unit vector $\eta \in \cH$ such that $\eta \perp \xi$ and $\eta \perp \zeta$. From the preceding paragraph, it follows that $\|\Phi (E_\xi)\|=\|\Phi (E_\eta )\|=\|\Phi (E_\zeta)\|$. Set $\kappa = \|\Phi (E_{\xi_0})\|$ for an arbitrary unit $\xi_0$.  This does not depend on the choice of $\xi_0$.

Next, we consider an arbitrary rank-one operator $A = \xi \otimes \zeta$, where $\xi,\zeta$ are unit vectors.  Let $\eta$ be a unit  vector in $\cH$ such that $\xi \perp \eta$ and $\zeta \perp \eta$. Since
\begin{align*}
(A+E_\eta)(A+E_\eta)^*
&= (\xi \otimes \zeta + \eta \otimes \eta)(\zeta \otimes \xi +\eta \otimes \eta ) \\
&= \xi \otimes \xi + \langle \zeta,\eta\rangle (\xi \otimes \eta) + \langle \eta,\zeta\rangle (\eta \otimes \xi) +  \eta \otimes \eta \\
&= E_\xi + E_\eta ,
\end{align*}
it follows that $|\Phi (A+E_\eta)^*|$ is a scalar multiple of a rank-two projection. Write $\Phi (A) = \xi_1 \otimes \zeta_1$ and $\Phi (E_\eta ) = \xi_2\otimes \zeta_2$.  Since $\Phi (A)\perps  \Phi (E_\eta)$ by $A\perps  E_\eta$, we have $|\Phi (A)^*| \perp_{BJ} |\Phi (E_\eta)^*|$, that is, $\xi_1 \otimes \xi_1 \perp_{BJ} \xi_2 \otimes \xi_2$ and $\xi_1 \perp \xi_2$. Then, as in the first paragraph, we can show that $\|\xi_1\|=\|\xi_2\|$. Thus, $\|\Phi (A)\|=\|\xi_1\|=\|\xi_2\|=\|\Phi (E_\eta)\|=\kappa$.
\end{proof}
\end{lemma}
%%%%%%%%%%%%%%%%%%%%%%%%%%%%%%%%%%%%%%%%%%%%%%%%%%%%%%%%%

%%%%%%%%%%%%%%%%%%%%%%%%%%%%%%%%%%%%%%%%%%%%%%%%%%%%%%%%%
\begin{lemma}\label{cpt-isometry}

Let $\A$ and $\B$ be compact $C^*$-algebras, and let $\Phi \colon   \A \to \B$ be a linear strong BJ isomorphism.
 Then, $\Phi$ is a scalar multiple of an isometry.
\begin{proof}
By Theorem~\ref{nonlinear-1}, we may assume that $\A = \B = \{ P_\lambda : \lambda \in \Lambda \}' \cap \cK (\cH)$ and $\Phi ( \sum_{\lambda \in \Lambda_0}\A P_\lambda ) = \sum_{\lambda \in \Lambda_0}\A P_\lambda$ for each $\lambda_0 \subset \Lambda$, where $\cH$ be a Hilbert space and $(P_\lambda )_{\lambda \in \Lambda} \subset \cB (\cH)$ is an orthogonal family of projections with sum $I$. In this case, Lemma~\ref{rank-one-norm} generates a positive number $\kappa$ such that $\|\Phi (A)\|=\kappa$ whenever $A \in \A$ is a rank-one operator with $\|A\|=1$. Let $A$ be an arbitrary norm-one element of $\A$, and let $A=|A^*|U$ be the polar decomposition for $A$.
 The double commutant theorem ensures that $U \in \{ P_\lambda : \lambda \in \Lambda \}'$.
 Since $\|A\|= \sup_\lambda \|AP_\lambda\|$ and $\sum_\lambda AP_\lambda$ converges to (a compact operator) $A$ in the norm topology, we can find a $\lambda \in \Lambda$ such that $\|AP_\lambda\|=\|A\|=1$.
 Let $\xi$ be a unit vector in $M_{|A^*P_\lambda|}=M_{|A^*|P_\lambda} \subset P_\lambda (\cH) \cap |A^*|(\cH)$. We note that $U$ is a partial isometry from $\overline{|A|(\cH)}$ onto $\overline{A(\cH)}=\overline{|A^*|(\cH)}$. Set $E$ be the rank-one projection onto $\mathbb{C}\xi$. We note that $E = \xi \otimes \xi \in \A$ and $EU \in \A$. This is because $\cK (\cH)$ is an ideal of $\cB (\cH)$. Moreover, we get $\|EU\|=1$ and $EU\perps  A-EU$. Indeed, since $UU^\ast$ is a projection onto $\overline{|A^*|(\cH)}\ni\xi$ we have $(EU)(EU)^* = EUU^*E = E$   and, since $AU^*=|A^*|$, further
\[
(A-EU)(A-EU)^* = (|A^*|-E)^2.
\]
 From this and $(|A^*|-E)\xi=\xi-\xi=0$, we obtain $|(EU)^*|\perps  |(A-EU)^*|$ and $EU\perps  A-EU$. Meanwhile, it follows from $M_{|(EU)^*|} = E(\cH) \subset M_{|A^*|}$ that $R_{EU}^s \subset R_A^s$ by Lemma~\ref{Rset-cpt}.

Now, we know that $\Phi (EU)$ is rank-one, $\Phi (EU)\perps  \Phi (A-EU)$ and $R_{\Phi (EU)} \subset R_{\Phi (A)}$. Set $\Phi (EU) = \zeta \otimes \eta$, where $\|\zeta \|=1$. Since $|\Phi (EU)^*| = \zeta \otimes \zeta$, again by Lemma~\ref{Rset-cpt}, we obtain $\mathbb{C}\zeta = M_{|\Phi (EU)^*|} \subset M_{|\Phi (A)^*|}$. Since $|\Phi (EU)^*|\perps  |\Phi (A-EU)^*|$, it follows that
\begin{align*}
0
&=\omega_\zeta (\Phi (A-EU)\Phi (A-EU)^*) \\
&= \omega_\zeta ((\Phi (A)-\zeta \otimes \eta)(\Phi (A)^*-\eta \otimes \zeta)) \\
&= \omega_\zeta (\Phi (A)\Phi (A)^* -\Phi (A)(\eta \otimes \zeta) -(\zeta \otimes \eta)\Phi (A)^* +\|\eta\|^2\zeta \otimes \zeta) \\
&= \|\Phi (A)^*\zeta \|^2 - \langle \eta,\Phi (A)^* \zeta \rangle - \langle \Phi (A)^*\zeta,\eta \rangle +\|\eta\|^2 \\
&= \langle \Phi (A)^*\zeta-\eta,\Phi (A)^*\zeta -\eta \rangle .
\end{align*}
Therefore, $\Phi (A^*)\zeta = \eta$ and
\[
\|\Phi (A)^*\|^2= \langle |\Phi (A)^*|\zeta,|\Phi (A)^*|\zeta \rangle  = \|\Phi (A)^*\zeta\|^2 = \|\eta\|^2
\]
by $\zeta \in M_{|\Phi (A)^*|}$, which implies that
\[
\|\Phi (A)\|=\|\Phi (A)^*\|=\|\eta\|=\|\Phi (EU)\|=\kappa ,
\]
where $\kappa$ is the value introduced in Lemma~\ref{rank-one-norm}. This proves that $\kappa^{-1}\Phi$ is an isometry.
\end{proof}
\end{lemma}
%%%%%%%%%%%%%%%%%%%%%%%%%%%%%%%%%%%%%%%%%%%%%%%%%%%%%%%%%

%%%%%%%%%%%%%%%%%%%%%%%%%%%%%%%%%%%%%%%%%%%%%%%%%%%%%%%%%

\begin{corollary}\label{lin-main}
Let $\A$ and $\B$ be compact $C^*$-algebras, and let $\Phi \colon   \A \to \B$ be a linear strong  BJ isomorphism.
 Then, there exist  nonzero scalars $\alpha,\beta \in \mathbb{C}$, a $*$-isomorphism $\Psi \colon   \A \to \B$ and  unitary elements $V \in \mathcal{M}(\A)$ and $W\in \mathcal{M}(\B)$ such that
$$\Phi (x)=\alpha\Psi(Vx)=\beta W\Psi(x).$$
\begin{proof}
Combine Theorem~\ref{cpt-isometry} with  Theorem~\ref{linear-isometry}.
\end{proof}
\end{corollary}
%%%%%%%%%%%%%%%%%%%%%%%%%%%%%%%%%%%%%%%%%%%%%%%%%%%%%%%%%

%%%%%%%%%%%%%%%%%%%%%%%%%%%%%%%%%%%%%%%%%%%%%%%%%%%%%%%%%%

%%%%%%%%%%%%%%%%%%%%%%%%%%%%%%%%%%%%%%%%%%%%%%%%%%%%%%%%%%
%%%%%%%%%%%%%%%%%%%%%%%%%%%%%%%%%%%%%%%%%%%%%%%%%%%%%%%%%%
\section{Concluding Remarks}
In this section we will translate previous results into operator-theoretic language. We will also improve them in case a compact $C^*$-algebra $\A$ contains blocks of small size. We always view a compact $C^\ast$-algebra as a $c_0$-direct sum
$$\A=\bigoplus_{\lambda\in\Lambda}\cK(\cH_{\lambda})$$ and embedded into $\cK(\cH)$ where  $\cH=\bigoplus_{\lambda\in\Lambda}\cH_\lambda$. Recall that every $A\in\A$ has a singular value decomposition (SVD for short)
$$A=\sum\sigma_i(A)\xi_i\otimes\zeta_i;\qquad \|A\|=\sigma_1(A)\ge \sigma_2(A)\ge\dots\ge0,$$
where $(\xi_\lambda)_\lambda$ and $(\zeta_\lambda)_\lambda$ are two orthonormal bases of $\cH$ and where $\sigma_i(A)$ converge to $0$; moreover, each pair $(\xi_i,\zeta_i)$ belongs to $\cH_\lambda$ for some index $\lambda$ (c.f.~Remark~\ref{rem:SVD-blockwise}). Also, $M_A=\mathrm{span}\{\zeta_1,\dots,\zeta_k\}$ where $\sigma_1(A)=\dots=\sigma_k(A)>\sigma_{k+1}(A)$ and $|A^\ast|=\sum\sigma_i(A) \xi_i\otimes\xi_i$, hence $M_{|A^\ast|}=AM_A$.\medskip

We start with  characterization of strong BJ orthogonality between compact operators.
\begin{lemma} Let $A,B\in \A=\bigoplus\cK(\cH_\lambda)$. Then, $A\perps B$ if and only if there exists an index $\lambda$ and a  unit vector $\zeta\in\cH_\lambda$ such that $\|A\zeta\|=\|A\|$ and $B^\ast A\zeta=0$.
\end{lemma}
\begin{proof} $\|A\zeta\|=\|A\|$ if and only if $\xi:=A\zeta\in AM_A=M_{|A^\ast|}$ or, equivalently, $\omega_{\zeta}(|A^\ast|)=\langle |A^\ast|\xi,\xi\rangle=\|A\|$. The rest follows from Lemmas~\ref{cpt-pure-state} and \ref{sBJ-pure}.
\end{proof}

Next is a restatement (and extension to  all Hilbert spaces) of Theorem~\ref{nonlinear-2} and Remark~\ref{rem:nonlinearBJpreservers-on-K(H)}. Let $R(A):=A(\cH)$ be the range of operator $A$. We say $\Phi\colon\cK(\cH)\to\cK(\cH)$ has property ${\mathcal P}$ if
$$  BM_B=AM_A\ \text{and}\ \overline{R(B)}=\overline{R(A)},\ \text{where}\ B=\Phi(A).
$$
Observe that this is equivalent to  the two equalities in~\eqref{eq:PropertyP}. We  also let $\PP(\cH):=\{[x]: x\in\cH\setminus\{0\}\}$ be the projective space over $\cH$, and declare projective points $[x],[y]$ orthogonal if $\langle x,y\rangle=0$.
 \begin{theorem}\label{thm:singleblock-a}
    Let $\Phi\colon\cK(\cH)\to \cK(\cH)$ be a strong BJ isomorphism.
    \begin{itemize}
    \item[(i)] If $\dim \cH\ge 3$, then there exists a  unitary or anti-unitary operator $U\colon \cH\to \cH$ such that $U^{-1}\circ \Phi$ has property  $\mathcal{P}$.
\item[(ii)]    If $\dim \cH=2$, then there  exist a bijection $\phi\colon\PP(\cH)\to\PP(\cH)$ which maps orthogonal projective points onto orthogonal projective points such that
 $$     B M_{B} = \phi(A M_A)\quad \hbox{ and } \quad    R(B)=\phi(R(A)),\ \text{where}\ B=\Phi(A).$$
\item[(iii)]    If $\dim \cH=1$, then any bijection which fixes $0$ is a strong BJ isomorphism on $\cK(\cH)$.
    \end{itemize}
\end{theorem}
\begin{proof}[Sketch of the proof]
    The arguments from the proof of Theorem~\ref{nonlinear-2} work regardless of the dimension of $\cH$ and show that $\Phi$ induces a bijective orthogonality-preserving map $\phi$ on $\PP(\cH)$ (where, with a slight abuse, we identified rank-one projection $\xi\otimes\xi\in\PP_1(\cH)$ with its range, $\CC\xi\in\PP(\cH)$).  If $\dim\cH\le 2$,   Uhlhorn's theorem is no longer available and we simply replace $M_{|\Phi(E)^\ast|}=UM_{E}$  and $\ker |\Phi(E)^\ast|=U(\ker E)$, $E=\xi\otimes\xi$  with an  equivalent statement $\Phi(E)M_{\Phi(E)}= \phi(EM_E)$ and $R(\Phi(E))=\phi(R(E))$. Following Example~\ref{exa:wild} we partition $\cK(\cH)$ into disjoint union of  sets $C(M,N)$, and define a bijection $\Psi\colon \cK(\cH)\to\cK(\cH)$ by
    $\Psi(0)=0$ and  mapping $C(M,N)$ bijectively onto $C(\phi(M),\phi(N))$ (with convention $\phi(0)=0$ and $\phi(\cH)=\cH$). This $\Psi$ is a strong BJ isomorphism, and a strong BJ isomorphism  $\Phi':=\Psi^{-1}\circ\Phi$ will satisfy $\|\Phi'(E)\|^{-1}|\Phi'(E)^\ast|=|E|$   for every rank-one projection  $E$. The rest follows arguments of the proof of Theorem~\ref{nonlinear-2}  with $U=I$.
\end{proof}
\begin{remark}\label{rem:examples-a}
(a) Alternative way to  obtain (possibly nonbijective) maps satisfying property ${\mathcal P}$ is with the help of SVD: Choose  any  operator  written in its SVD,
$$A=\sum \sigma_i(A) \xi_i\otimes \zeta_i;\qquad \sigma_1(A)=\dots=\sigma_k(A)>\sigma_{k+1}(A)\ge \dots\ge0;$$
 select unitaries $U_A,V_A$ which depend on $A$, a nonzero scalar $\gamma_A$,  and select a $c_0$ sequence $\tau_{k+1}(A), \tau_{k+2}(A),\ldots <\sigma_1(A)$, which also depends on $A$. Assume in addition that $AM_A=\mathrm{span}\, \{\xi_1,\dots,\xi_k\}$ and $R(A)\ominus A M_A=\mathrm{span}\,\{\xi_{k+1},\xi_{k+2},\dots\}$ are both invariant for $U_A$.  Then, the map
$$A\mapsto \gamma_A U_A\Bigl(\sigma_1(A) \sum_1^k \xi_i\otimes \zeta_i +\sum_{k+1}^\infty \tau_i (A)\xi_i\otimes \zeta_i\Bigr)V_A^\ast  $$
is the most general form of those  maps. It can more shortly be written as
$$A\mapsto \gamma_A U_AP_AAV_A^\ast,$$ where $P_A$ is a diagonal positive-definite map on orthonormal basis $(\xi_i)_i$ such that $P_A(\sigma_i(A)\xi_i)=\tau_i(A) \xi_i$ and $\tau_{k+i}(A)<\tau_k(A)=\dots=\tau_1(A)=\sigma_1(A)$. We warn that $\sum\tau_i(A)\xi_i\otimes \zeta_i$  might not be a SVD, since we allow also nonmonotonic sequences for  $\tau_i(A)$.

(b) We can also list all the strong BJ isomorphism  in the case of two-dimensional space $\cH$. The orthogonality preserving bijections $\phi\colon\PP(\cH)\to\PP(\cH)$ are simply the  permutations of the partition
$$\PP(\cH)=\bigcup\{[(\alpha,\beta)],[(-\bar{\beta},\bar{\alpha})]\},\qquad (\alpha,\beta)\in\CC^2\setminus\{0\}.$$ Given  any such  orthogonality preserving bijection $\phi\colon\PP(\cH)\to\PP(\cH)$, let $\varphi_\phi$ be a homogeneous  map from  $\cH$ onto itself,  such that $\|\varphi_\phi(\xi)\|=\|\xi\|$ (i.e., it maps a unit ball of $\cH$ onto itself) and $[\varphi_\phi(\xi/\|\xi\|)]=\phi([\xi])$. Define $\Phi_\phi(0)=0$ and extend it to
a strong BJ isomorphism (it maps $C(M,N)$ onto $C(\phi(M),\phi(N))$; see Example~\ref{exa:wild}) on $\cK(\cH)$, via
\begin{equation}\label{eq:induced=by-phi}
\Phi_\phi\colon A=\sum \sigma_i(A) \xi_i\otimes\zeta_i\mapsto \sum \sigma_i(A) \varphi_\phi(\xi_i)\otimes\zeta_i.
\end{equation}
We may further compose $\Phi_\phi$ with a map $A\mapsto \gamma_A U_AP_AAV_A^\ast$ as in part (a); one checks that this  composition is (an arbitrary) bijection on  the set of scalar multiples of unitaries, and maps nonzero $A=\sum_1^2\sigma_1(A) \xi_i\otimes \zeta_i$, written in its SVD, with $\sigma_2(A)<\sigma_1(A)$ into
$$\Phi_\phi(A)= \gamma_A\bigl( \sigma_1(A) \varphi_\phi(\xi_1)\otimes \zeta_1+ \tau_2(A) \varphi_\phi(\xi_2)\otimes \zeta_2\bigr) V_A,\quad \sigma_1(A)>\sigma_2(A)$$
for some unitary $V_A$, a nonzero scalar $\gamma_A$ and $\tau_2(A)<\sigma_1(A)$.
These are all the maps from (ii) of Theorem~\ref{eq:PropertyP}.
\end{remark}
We call a map $\Phi_\phi$ from \eqref{eq:induced=by-phi} to be \emph{induced by orthogonality-preserving bijection~$\varphi$}; observe that $\|\Phi_\phi(A)\|=\|A\|$.

\begin{theorem}\label{thm:general-a} Let $\A=\bigoplus \cK(\cH_\lambda)$ be a compact $C^*$-algebra and $\Phi\colon\A\to\A$ a strong BJ isomorphism. Then, $\Phi$ is composed of
\begin{itemize}
    \item[(i)] A  permutation of blocks which respects their dimensions.
\item[(ii)] A map which is a multiplication by unitaries  on some blocks and multiplication by anti-unitaries on other blocks.
\item[(iii)] A map which is identity on all blocks with $\dim  \cH_\lambda\neq2$ and is induced by some orthogonality-preserving bijections $\phi_\lambda\colon\PP(\cH_\lambda)\to \PP(\cH_\lambda)$ on  blocks with $\dim \cH_\lambda=2$.
  \item[(iv)] A map with property ${\mathcal P}$ which maps norm-achieving blocks to norm-achieving blocks.
\end{itemize}
\end{theorem}
\begin{proof}[Sketch of the proof] Following the proof (and notations) of Theorem~\ref{nonlinear-1} there exists a permutation of indices $\varphi\colon\Lambda\to\Lambda$ such that $\dim\cH_\lambda=\dim\cH_{\varphi(\lambda)}$ and $\Lambda_{\Phi(A)}=\varphi(\Lambda_A)$ (c.f.~Lemma~\ref{central-ortho} for notation~$\Lambda_A$) for nonzero $A\in\A$. We extend  $\varphi$ to a linear $\ast$-isomorphism on $\A$ by
$$\varphi\colon A=\bigoplus A_\lambda\mapsto \bigoplus A_{\varphi(\lambda)};$$
being a $\ast$-isomorphism it preserves strong BJ orthogonality in both directions. Then, $\Phi':=\varphi^{-1}\circ \Phi$ is a strong BJ isomorphism and $ \Phi'( \sum_{\lambda \in \Lambda_0}\A P_\lambda ) = \sum_{\lambda \in \Lambda_0}\A P_\lambda$ for each $\Lambda_0 \subset \Lambda$.

The restriction of $\Phi'$ to a sum of  nonempty collection of blocks in $\A$ remains to be a strong BJ isomorphism by Lemma~\ref{reduction}; its structure on individual blocks $\cK(\cH_\lambda)$ is given in Theorem~\ref{thm:singleblock-a} and Remark~\ref{rem:examples-a} and takes the form (ii)-(iv), restricted to $\cK(\cH_\lambda)$. The maps in (ii)--(iii) do not change the norm of individual blocks and are strong BJ  isomorphisms (again see Example~\ref{exa:wild}); by composing  $\Phi'$ with their inverses we obtain  a strong BJ isomorphism $\Phi''$ which has property ${\mathcal P}$ when restricted  to  individual blocks and does not change the norm of blocks. Hence, by applying Lemmas~\ref{Rset-cpt}  and \ref{Lset-cpt} on each norm-attaining block of $B\in\A$ and on $B$, it has globally the property ${\mathcal P}$ and takes the form (iv).
\end{proof}
Linear strong BJ isomorphisms have a simpler structure.
\begin{theorem}\label{thm:generallinear} Let $\A=\bigoplus \cK(\cH_\lambda)$ be a compact $C^*$-algebra and $\Phi\colon\A\to\A$ be a linear  strong BJ isomorphism. Then, there exists a scalar $\alpha\neq0$ such that $\frac{1}{\alpha}\Phi$ is composed of
\begin{itemize}
    \item[(i)] A  permutation of blocks which respects their dimensions.
\item[(ii)] A map $X\mapsto UXV^\ast$ for some unitaries $U=\bigoplus U_\lambda$ and $V=\bigoplus V_\lambda$ with  $U_\lambda,V_\lambda\in\cB(\cH_\lambda)$.
\end{itemize}
\end{theorem}
\begin{proof}[Sketch of the proof]
 By Corollary~\ref{lin-main} we can assume $\Phi$ is an isometry.
    As in the proof of Theorem~\ref{thm:general-a} we  compose $\Phi$ by a suitable map of the form (i) to obtain a strong BJ isomorphism  $\Phi'$ with the property  $ \Phi'( \sum_{\lambda \in \Lambda_0}\A P_\lambda ) = \sum_{\lambda \in \Lambda_0}\A P_\lambda$ for each $\Lambda_0 \subset \Lambda$. By Corollary~\ref{lin-main} this map is a multiplication of a unitary  and a $\ast$-isomorphism, moreover, it preserves rank-one by Lemma~\ref{rank-n}. Hence, its restriction on individual block $\cK(\cH_\lambda)\subseteq\A$ takes the form $X_\lambda\mapsto W_\lambda V_\lambda X_\lambda V_\lambda^\ast$ for some unitaries $W_\lambda,V_\lambda\in\cB(\cH_\lambda)$,  (see~\cite[Corollary~5.43]{Douglas1972}) and the result follows by linearity and continuity.
\end{proof}

%%%%%%%%%%%%%%%%%%%%%%%%%%%%%%%%%%%%%%%%%%%%%%%%%%%%%%%%%%
%%%%%%%%%%%%%%%%%%%%%%%%%%%%%%%%%%%%%%%%%%%%%%%%%%%%%%%%%%
\end{document}